\documentclass[11pt]{article}

\usepackage{framed}

\usepackage{geometry}           
\usepackage{amsthm}             
\usepackage{graphicx}           
\usepackage{amsmath}            
\usepackage{amssymb}            
\usepackage{float}              
\usepackage{verbatim}           
\usepackage{url}           	
\usepackage{xspace}		
\usepackage{pdfpages}           
\usepackage[sort,numbers]{natbib} 
\usepackage[breaklinks=true, colorlinks=true,
linkcolor=blue, citecolor=blue, urlcolor=blue]{hyperref} 
\usepackage{doi}
\usepackage{enumitem} 
\usepackage{subfigure}
\usepackage{epstopdf}
\usepackage{hyperref}

\usepackage{algorithm,algorithmic}
\newcommand{\balgorithm}  {\begin{algorithm}}
\newcommand{\ealgorithm}  {\end{algorithm}}
\newcommand{\balgorithmic}{\begin{algorithmic}}
\newcommand{\ealgorithmic}{\end{algorithmic}}

\usepackage{pst-tree}

\graphicspath{{images/}}

\newcommand{\R}{\mbox{${\mathbb R}$}}           
\newcommand{\E}{\mathbb E}           

\newtheorem{theorem}{Theorem}[section]

\newtheorem{assumption}[theorem]{Assumption}

\newcommand{\RCNote}[1]           
{\textcolor{red}{#1}\marginpar[\textbf{$\Longrightarrow$}]
{\textbf{$\Longleftarrow$}}}

\newcommand{\SWNote}[1]           
{\textcolor{blue}{#1}\marginpar[\textbf{$\Longrightarrow$}]
{\textbf{$\Longleftarrow$}}}

\makeatletter
\def\namedlabel#1#2{\begingroup
    #2%
    \def\@currentlabel{#2}%
    \phantomsection\label{#1}\endgroup
}
\makeatother

\newtheoremstyle{mytheoremstyle} 
    {\topsep}                    
    {\topsep}                    
    {\normalfont}                   
    {}                           
    {\scshape}                   
    {.}                          
    {.5em}                       
    {}  

\theoremstyle{mytheoremstyle}
\newtheorem{lemma}[theorem]{Lemma}

\numberwithin{equation}{section}
\numberwithin{figure}{section}
\numberwithin{table}{section}

\begin{document}


\begin{center}
{\Large\textbf{Randomized Derivative-Free Optimization of Noisy
Convex Functions}\footnotemark [1]}
%

\vspace{0.15in}

\textbf{Ruobing Chen\footnotemark [2] 
\quad
Stefan M. Wild\footnotemark [3]}
\vspace{0.15in}

\today
\end{center}

\footnotetext [1] {This material was based upon work supported by the  U.S.
Department of Energy, Office of Science, Office of Advanced Scientific Computing
Research, under Contract DE-AC02-06CH11357. }

\footnotetext [2] {Data Mining Services and Solutions, Bosch Research and Technology Center, Palo Alto, CA 94304.}

\footnotetext [3] {Mathematics and Computer Science Division, Argonne
  National Laboratory, Argonne, IL 60439.  }

\bigskip


\begin{abstract}
We propose {\sf STARS}, a randomized derivative-free algorithm for unconstrained 
optimization when the function evaluations are contaminated with random noise. 
{\sf STARS} takes dynamic, noise-adjusted smoothing stepsizes that minimize
the 
least-squares error between the true directional derivative of a noisy function 
and its finite difference approximation. We provide a convergence rate analysis
of 
{\sf STARS} for solving convex problems with additive or multiplicative noise. 
Experimental results show that (1) {\sf 
STARS} exhibits noise-invariant behavior with respect to different levels of 
stochastic noise; (2) the practical performance of {\sf STARS} in terms of
solution 
accuracy and convergence rate is significantly better than that indicated by the 
theoretical result; and (3) {\sf STARS} outperforms a selection of randomized 
zero-order methods on both additive- and multiplicative-noisy functions.
\end{abstract}


\section{Introduction}
\label{sec:intro}
We propose {\sf STARS}, a randomized derivative-free algorithm for unconstrained 
optimization when the function evaluations are contaminated with random noise. 
Formally, we address the stochastic optimization problem 
\begin{equation}
\underset{x\in \R^n}{\min}   f(x)=\E_\xi \left[
\tilde{f}(x;\xi)   \right], 
 \label{eqn.prob.chp3}
\end{equation}
where the objective $f(x)$ is assumed to be differentiable but is available
only through noisy realizations $\tilde{f}(x;\xi)$. In particular, 
although our analysis will at times assume that the 
gradient of the objective function $f(x)$ exist and be Lipschitz
continuous, 
we assume that direct evaluation of these derivatives is impossible. Of special 
interest to this work are situations when derivatives are unavailable or
unreliable because of  
stochastic noise in the objective function evaluations. This type of noise
introduces the dependence on the random variable $\xi$ in 
(\ref{eqn.prob.chp3}) and may 
arise if random fluctuations or measurement errors occur in a 
simulation producing the objective $f$. In addition to stochastic and 
Monte Carlo simulations, this stochastic noise can also be used to model the
variations in 
iterative or adaptive simulations resulting from finite-precision 
calculations and specification of internal tolerances \cite{more2014nd}.

Various methods have been designed for optimizing problems with 
noisy function evaluations. One such class of methods, dating back half a 
century, are \textit{randomized search methods} \cite{matyas1965}. Unlike 
classical, deterministic direct search methods \cite{VTorczon_1991, 
VTorczon_1997, RMLewis_VTorczon_MTrosset_2000, MAAbramson_CAudet_2006, 
CAudet_JEDennis_2006, MAAbramson_etal_2008}, randomized search methods 
attempt to accelerate the optimization by using random vectors as search 
directions. These randomized schemes share a simple basic framework, allow 
fast initialization, and have shown promise for solving large-scale
derivative-free problems \cite{Stich2011,Lan2012}.
Furthermore, optimization folklore and intuition suggest that these randomized steps should 
make the methods less sensitive to modeling errors and ``noise'' in the general 
sense; we will systematically revisit such intuition in our computational 
experiments. 

Recent works have addressed the special cases of zero-order minimization of
convex functions with additive noise. For instance, Agarwahl et
al.~\cite{NIPS2011_4475}
utilize a bandit feedback model, but the regret bound depends on a term of
order $n^{16}$.
Recht et al.~\cite{NIPS2012_4509} consider a coordinate descent approach
combined with
an approximate line search that is robust to noise, but only theoretical bounds
are provided. Moreover, the situation
where the noise is nonstationary (for example, varying relative to the
objective function) remains largely unstudied.

Our approach is inspired by the recent work of Nesterov \cite{Nest2011}, which 
established complexity bounds for convergence of random derivative-free methods 
for convex and nonconvex functions. Such methods work by iteratively moving 
along directions sampled from a normal distribution surrounding the current 
position. The conclusions are true for both the smooth and nonsmooth 
Lipschitz-continuous cases. Different improvements of these random search ideas
appear in 
the latest literature.  For instance, Stich et al.~\cite{Stich2011} give 
convergence rates for an algorithm where the search directions are 
uniformly distributed random vectors in a hypersphere and the stepsizes are 
determined by a line-search procedure. Incorporating the Gaussian smoothing 
technique of Nesterov \cite{Nest2011}, Ghadimi and Lan \cite{Lan2012} present
a randomized 
derivative-free  method for stochastic optimization and show that the
iteration complexity of their algorithm 
improves Nesterov's result by a factor of order $n$ in the smooth, convex 
case. 
Although complexity bounds are readily available for these randomized
algorithms, the practical 
usefulness of these algorithms and their potential for dealing with noisy 
functions have been relatively unexplored.

In this paper, we address ways in which a randomized method can benefit 
from careful choices of noise-adjusted smoothing stepsizes. We propose a new 
algorithm, {\sf STARS}, short for STepsize Approximation in Random Search. The
choice of stepsize 
work is greatly motivated by Mor{\'e} and Wild's recent work on estimating 
computational noise \cite{JMSMW11} and derivatives of noisy 
simulations \cite{More2012}. {\sf STARS} takes dynamically changing smoothing 
stepsizes that minimize the least-squares error between the true directional 
derivative of a noisy function and its finite-difference approximation. We 
provide a convergence rate analysis of {\sf STARS} for solving convex problems 
with both additive and multiplicative stochastic noise. With nonrestrictive 
assumptions 
about the noise, {\sf STARS} enjoys a convergence rate for noisy convex
functions 
identical to that of Nesterov's random search method for smooth convex 
functions. 

The second contribution of our work is a numerical study of {\sf 
STARS}. Our experimental results illustrate that (1) the performance of {\sf
STARS} exhibits little variability with respect to different levels of
stochastic noise; 
(2) the practical performance of {\sf STARS} in terms of solution accuracy and 
convergence rate is often significantly better than that indicated by the
worst-case, theoretical bounds; and (3) {\sf STARS} outperforms a selection of
randomized zero-order methods 
on both additive- and multiplicative-noise problems. 

The remainder of this paper is organized as follows. In Section~\ref{sec:ro} 
we review basic assumptions about the noisy function setting and 
results on Gaussian smoothing. Section~\ref{sec:STARS} presents the new 
{\sf STARS} algorithm. In
Sections~\ref{sec:additive.noise}~and~\ref{sec:mul.noise}, a convergence rate
analysis is provided for solving convex 
problems with additive noise and multiplicative noise, respectively.  
Section~\ref{sec:experiments} presents an empirical study of {\sf STARS} on
popular test problems by examining the performance relative to both the
theoretical bounds and other randomized derivative-free solvers. 

\section{Randomized Optimization Method Preliminaries}
\label{sec:ro}

One of the earliest randomized algorithms for the nonlinear, deterministic 
optimization problem 
\begin{equation}
 \min_{x \in \R^n} f(x),
 \label{eq:detprob}
\end{equation}
where the objective function $f$ is assumed to be differentiable but 
evaluations of the gradient $\nabla f$ are not employed by the algorithm, is
attributed to Matyas \cite{matyas1965}.
Matyas introduced a \textit{random optimization approach} that,
at every iteration $k$, randomly samples a point $x_{+}$ from a Gaussian
distribution centered on the current point $x_k$. The function is evaluated
at $x_+=x_k+u_k$, and the iterate is updated depending on whether decrease has
been seen:
\begin{equation*}
 x_{k+1} = \begin{cases}
            x_+ & \mbox{if } f(x_{+})<f(x_k) \\
            x_k & \mbox{otherwise.}
           \end{cases}
\end{equation*}

Polyak \cite{polyak1987} improved
this scheme by describing stepsize rules for iterates of the form 
\begin{equation}\label{eq:poljak}
x_{k+1} = x_k - h_k \dfrac{ f(x_k+\mu_k u_k)-f(x_k) }{\mu_k}u_k,
\end{equation}
where $h_k>0$ is the stepsize, $\mu_k>0$ is called the smoothing stepsize,
and $u_k\in \R^n$ is a random direction.

Recently, Nesterov \cite{Nest2011} has revived interest in Poljak-like schemes
by showing that {Gaussian} directions $u\in \R^n$ 
allow one to benefit from properties of a Gaussian-smoothed version of the
 function $f$, 
 \begin{equation}
  f_\mu(x)=\E_u[f(x+\mu u)],
  \label{eq:smoothedf}
 \end{equation}
where $\mu >0$ is again the smoothing stepsize and where
we
have made explicit that the expectation is being taken with respect to the
random vector $u$. 

Before proceeding, we review additional notation and results concerning Gaussian
smoothing.

\subsection{Notation}

We say that a function $f\in  \mathcal{C}^{0,0}(\mathbb{R}^n)$ if
$f:\mathbb{R}^n \mapsto \mathbb{R}$ is continuous and there exists a
constant $L_0$ such
that
$$ | f(x)- f(y) |\le L_0 \|x-y\|,\quad \forall x,y\in \R^n,$$
where $\|\cdot\|$ denotes the Euclidean norm. 
We say that $f\in \mathcal{C}^{1,1}(\mathbb{R}^n)$ if $f:\mathbb{R}^n
\mapsto \mathbb{R}$ is continuously differentiable and there exists a constant
$L_1$ such
that
\begin{equation} 
 \| \nabla f(x)-\nabla f(y) \|\le L_1 \|x-y\| \quad \forall x,y\in \R^n.
\label{eq:L1first}
\end{equation}
Equation~(\ref{eq:L1first}) is equivalent to
\begin{equation}
| f(y)-f(x) -\langle \nabla f(x), y-x \rangle  |   \le \dfrac{L_1}{2} \| x-y 
\|^2 \quad \forall x,y\in \R^n,  
\label{eq:L1second}
\end{equation}
where $\langle\cdot,\cdot\rangle$ denotes the Euclidean inner product.

Similarly, if $x^*$ is a global minimizer of $f\in
\mathcal{C}^{1,1}(\mathbb{R}^n)$, then (\ref{eq:L1second}) implies that
\begin{equation} 
\| \nabla f(x) \|^2 \le 2L_1(f(x)-f(x^*)) \quad \forall x\in \R^n.  
\label{eq:L1third}
\end{equation}
We recall that a differentiable function $f$ is convex if 
\begin{equation} 
f(y)\ge f(x)+ \langle  \nabla f(x),y-x \rangle  \quad \forall x,y\in \R^n. 
\label{eq:convexity}
\end{equation}

\subsection{Gaussian Smoothing}

We now examine properties of the Gaussian approximation of $f$ in
(\ref{eq:smoothedf}). For $\mu\neq 0$, we let $g_\mu(x)$ be the
first-order-difference approximation of the derivative of $f(x)$ in the
direction $u\in \R^n$,
\[g_\mu(x) = \frac{f(x+\mu u)-f(x)}{\mu}u,\]
where the nontrivial direction $u$ is implicitly assumed.
By $\nabla f_\mu(x)$ we denote the gradient (with respect to $x$) of the
Gaussian approximation in (\ref{eq:smoothedf}). 
For standard (mean zero, covariance $I_n$) Gaussian random vectors $u$ and a
scalar $p\geq 0$, we define
\begin{equation}
 M_p\equiv\E_{u}[\| u \|^p ] = 
\frac{1}{(2\pi)^{\frac{n}{2}}}
\int_{\mathbb{R}^n}  \| u \|^p  e^{-\frac{1}{2}\| u \|^2}du.
\label{eq:Mp}
\end{equation}

We summarize the relationships for Gaussian smoothing from 
\cite{Nest2011} upon which we will rely in the following lemma.

\begin{lemma}\label{lemma.nesterov.results}
Let $u\in \R^n$ be a normally distributed Gaussian vector. Then, the following
are true.
\begin{itemize}
\item[(a)] For $M_p$ defined in (\ref{eq:Mp}), we have 
\begin{eqnarray}
M_p&\le &n^{p/2}, \quad \mbox{ for } p\in [0,2], \quad \mbox{and}\label{eq:M_p.le2} \\
M_p &\le &(n+p)^{p/2}, \quad \mbox{ for } p> 2.\label{eq:M_p.ge2}
\end{eqnarray}

\item[(b)] If $f$ is convex, then
\begin{eqnarray}
f_\mu(x)&\ge& f(x) \quad \forall x\in \R^n. \label{eq.f_mu.1}
\end{eqnarray}
\item[(c)] 
If $f$ is convex and $f\in\mathcal{C}^{1,1}(\R^n)$, then
\begin{eqnarray}
|f_\mu(x)-f(x)|&\le& \frac{\mu^2}{2}L_1n \quad \forall x\in 
\R^n.\label{eq.f_mu.2}
\end{eqnarray}

\item[(d)] If $f$ is differentiable at $x$, then 
\begin{eqnarray}
\E_{u}[g_\mu(x)]&=&\nabla f_\mu(x) \quad \forall x\in \R^n.
\label{eqn.exp.of.g_mu}
\end{eqnarray}
\item[(e)] 
If $f$ is differentiable at $x$ and $f\in\mathcal{C}^{1,1}(\R^n)$, then
\begin{eqnarray}
\E_{u}[\|g_\mu(x)\|^2]&\le& 2(n+4)\|\nabla
f(x)\|^2+\frac{\mu^2}{2}L_1^2(n+6)^3\quad \forall x\in \R^n.
\label{eqn.exp.of.g_mu.square}
\end{eqnarray}
\end{itemize}
\end{lemma}

\section{The {\sf STARS} Algorithm}\label{sec:STARS}

The {\sf STARS} 
algorithm for
solving (\ref{eqn.prob.chp3}) while having access to the objective $f$ only
through its noisy version $\tilde{f}$ is summarized in
Algorithm~\ref{alg:STARS}.

\balgorithm[t]
\caption{({\sf STARS}: STep-size Approximation in Randomized Search)}
\label{alg:STARS}
 \balgorithmic[1]
\STATE Choose initial point $x_1$, iteration limit $N$, stepsizes
$\{h_k\}_{k\ge1}$. Evaluate the function at the initial point to obtain 
$\tilde{f}(x_1;\xi_0)$. Set $k \gets 1$.
\STATE Generate a random Gaussian vector $u_k$, and compute the smoothing
parameter $\mu_k$. \\
\STATE Evaluate the function value $\tilde{f}(x_k + \mu_ku_k;\xi_k)$.\\
\STATE Call the stochastic gradient-free oracle 
\begin{equation}
 s_{\mu_k} (x_k;u_k,\xi_k,\xi_{k-1}) =
\frac{\tilde{f}(x_k+\mu_ku_k;\xi_k)-\tilde{f}(x_k;\xi_{k-1})}{\mu_k}u_k.
\label{eq:gfo}
\end{equation}
\STATE Set  $x_{k+1}=x_k-h_k s_{\mu_k} (x_k;u_k,\xi_k,\xi_{k-1}) $.\\
\STATE Evaluate $\tilde{f}(x_{k+1};\xi_k)$, update $k\gets k+1$, and return to
Step 2.
  \ealgorithmic
\ealgorithm

In general, the Gaussian directions used by Algorithm~\ref{alg:STARS} can come
from general Gaussian directions (e.g., with the covariance informed by 
knowledge about the scaling or curvature of $f$). For simplicity of
exposition, however, we focus on standard Gaussian directions as formalized in 
Assumption \ref{assumption_u}.
The general case can be recovered by a change of variables with an appropriate
scaling of the Lipschitz constant(s).

\begin{assumption}[Assumption about direction $u$]\label{assumption_u} 
In each iteration $k$ of Algorithm~\ref{alg:STARS}, $u_k$ is a vector drawn
from a multivariate normal distribution with mean $0$ and covariance matrix
$I_n$; equivalently, each element of $u$ is independently and identically
distributed (i.i.d.) from a standard normal distribution, $\mathcal{N}(0,1)$.
\end{assumption}

What remains to be specified is the smoothing stepsize $\mu_k$. It is computed
by incorporating the noise information so that the approximation of the
directional derivative has minimum error. We address two types of noise:
\textit{additive noise} (Section~\ref{sec:additive.noise}) and
\textit{multiplicative noise} (Section~\ref{sec:mul.noise}). 
These two forms of how $\tilde{f}$ depends on the random variable
$\xi$ correspond to two ways that noise often enters a system. The following
sections provide near-optimal expressions for $\mu_k$ and a convergence rate
analysis for both cases.

Importantly, we note Algorithm~\ref{alg:STARS} allows the random variables
$\xi_k$ and $\xi_{k-1}$ used in (\ref{eq:gfo}) to be different from one
another. This generalization is in contrast to the stochastic optimization
methods examined in \cite{Nest2011}, where it is assumed the same random
variables are used in the smoothing calculation. 
This generalization does not affect the additive noise case, but will
complicate the multiplicative noise case.

\section{Additive Noise}\label{sec:additive.noise}

We first consider an \textit{additive noise} model for the stochastic 
objective function $\tilde{f}$:
\begin{equation}
\tilde{f}(x;\xi)=f(x)+\nu(x;\xi),
\label{eq:funcadd}
\end{equation}
where $f:\mathbb{R}^n \mapsto \mathbb{R}$ is a smooth, deterministic function,
$\xi \in \Xi$ is a random vector with probability distribution $P(\xi)$,
and $\nu(x;\xi)$ is the stochastic noise component. 

We make the following assumptions about $f$ and $\nu$.

\begin{assumption}[Assumption about $f$]\label{assumption_f_add} 
$f\in \mathcal{C}^{1,1}(\mathbb{R}^n)$ and $f$ is convex. 
\end{assumption}

\begin{assumption}[Assumption about additive $\nu$]\label{assumption_add}
\item[1.] For all $x\in \R^n$, $\nu$ is i.i.d. with bounded variance
$\sigma_a^2 = \mbox{Var}(\nu(x;\xi))>0$.

\item[2.]  For all $x\in \R^n$, the noise is unbiased; that is, $\E_\xi [\nu
(x;\xi)] =0$. 
\end{assumption}

We note that $\sigma_a^2$ is independent of $x$ since $\nu(x;\xi)$ is
identically distributed for all $x$. The second assumption is nonrestrictive,
since if $\E_\xi [\nu (x;\xi)] \neq 0$, we
could just redefine $f(x)$ to be $f(x)-\E_\xi [\nu (x;\xi)]$.

\subsection{Noise and Finite Differences}

Mor\'e and Wild \cite{More2012} introduce a way of computing the smoothing
stepsize $\mu$ that mitigates the effects of the noise in $\tilde{f}$ when
estimating a first-order directional directive. The method involves 
analyzing the expectation of the least-squared error between the
forward-difference approximation, $\frac{\tilde{f}(x+\mu u
;\xi_1)-\tilde{f}(x;\xi_2)}{\mu}$, and the directional derivative of the
smooth function, $\langle \nabla f(x),u\rangle $. The authors show that a
near-optimal $\mu$ can be computed in such a way that the expected error has the
tightest upper bound among all such values $\mu$. Inspired by their approach, we
consider the least-square error between $\frac{\tilde{f}(x+\mu
u;\xi_1)-\tilde{f}(x;\xi_{2})}{\mu}u$ and $\langle \nabla f(x),u\rangle u$. That
is, our goal is to find $\mu^*$ that minimizes an upper bound on $\E[
\mathcal{E}(\mu) ]$, where
\begin{equation*}
 \mathcal{E}(\mu)\equiv \mathcal{E}(\mu;x,u,\xi_1,\xi_2)=\left\|
\frac{\tilde{f}(x+\mu u;\xi_1)-\tilde{f}(x;\xi_{2})}{\mu}u-\langle \nabla
f(x),u\rangle u  \right\|^2.
\end{equation*}
We recall that $u$, $\xi_1$, and $\xi_2$ are independent random
variables. 


\begin{theorem}\label{thm.optstep.add}
Let Assumptions \ref{assumption_u}, \ref{assumption_f_add}, and
\ref{assumption_add} hold. 
If a smoothing stepsize is chosen as 
\begin{equation}\label{eqn.opt.mu}
\mu^* = \left[\frac{8\sigma_a^2n}{L_1^2(n+6)^3}\right]^{\frac{1}{4}}, 
\end{equation}
then for any $x\in \R^n$, we have 
\begin{eqnarray}\label{eqn.opt.error.bound.add}
\E_{u,\xi_1,\xi_2}[ \mathcal{E}(\mu^*) ] \le \sqrt{2}L_1\sigma_a \sqrt{n(n+6)^3}. 
\end{eqnarray}
\end{theorem}

\begin{proof}
Using (\ref{eq:funcadd}) and (\ref{eq:L1second}), we derive
\begin{eqnarray*}
\mathcal{E}(\mu)&\le& \left\|\frac{\nu(x+\mu u;\xi_1)-\nu(x;\xi_2)}{\mu} u +
\frac{\mu L_1}{2} \|u\|^2 u\right\|^2 \\
&\le& \left( \frac{\nu(x+\mu u;\xi_1)-\nu(x;\xi_2)}{\mu} + \frac{\mu L_1}{2}
\|u\|^2 \right)^2 \|u\|^2 . 
\end{eqnarray*}
Let $X =  \frac{\nu(x+\mu u;\xi_1)-\nu(x;\xi_2)}{\mu} + \frac{\mu L_1}{2}
\|u\|^2 $. 
By Assumption~\ref{assumption_add}, 
the expectation of $X$ with respect to $\xi_1$ and
$\xi_2$ is $\E_{\xi_1,\xi_2}[X]=\frac{\mu L_1}{2} \|u\|^2$, and the
corresponding variance is 
$\mbox{Var}(X)=\frac{2\sigma_a^2}{\mu^2}$. 
It then follows that
\begin{eqnarray*}
\E_{\xi_1,\xi_2}[X^2] =(\E_{\xi_1,\xi_2}[X])^2+ \mbox{Var}(X) = \frac{\mu^2 L_1^2}{4} \|u\|^4+\frac{2\sigma_a^2}{\mu^2}.
\end{eqnarray*}

Hence, taking the expectation of $\mathcal{E}(\mu)$ with respect to $u,\xi_1$,
and $\xi_2$ yields
\begin{eqnarray*}
\E_{u,\xi_1,\xi_2}[ \mathcal{E}(\mu) ]&\le&\E_u\left[  \E_{\xi_1,\xi_2}[X^2\|u\|^2]    \right]\\
&=&\E_u\left[   \frac{\mu^2 L_1^2}{4} \|u\|^6+\frac{2\sigma_a^2}{\mu^2}   
\|u\|^2   \right] .
\end{eqnarray*}
Using (\ref{eq:M_p.le2}) and (\ref{eq:M_p.ge2}), we can further derive
\begin{equation}
\E_{u,\xi_1,\xi_2}[ \mathcal{E}(\mu) ] \le  \frac{\mu^2
L_1^2}{4} (n+6)^3 + \frac{2\sigma_a^2}{\mu^2}n.
 \label{eq:absgenbound}
\end{equation}

The right-hand side of (\ref{eq:absgenbound}) is uniformly convex in $\mu$ and
has a global minimizer of 
\begin{equation*}
\mu^* = \left[\frac{8\sigma_a^2n}{L_1^2(n+6)^3}\right]^{\frac{1}{4}}, 
\end{equation*}
with the corresponding minimum value yielding (\ref{eqn.opt.error.bound.add}).
\end{proof}

\paragraph{Remarks:} 
\begin{itemize}
\item A key observation is that for a function $\tilde{f}(x;\xi)$ with additive
noise,
as long as the noise has a constant variance $\sigma_a>0$, the optimal choice of
the stepsize $\mu^*$ is independent of $x$.
\item Since the proof of Theorem \ref{thm.optstep.add} does not rely on the
convexity assumption about $f$, the error bound (\ref{eqn.opt.error.bound.add})
for
the finite-difference approximation also holds for the nonconvex case.
The convergence rate analysis for {\sf STARS} presented in the next section,
however, will assume convexity of $f$; the nonconvex case is out of the
scope of this paper but is of interest for future research.
\end{itemize}
\subsection{Convergence Rate Analysis}

We now examine the convergence rate of Algorithm~\ref{alg:STARS}
applied to the additive noise case of (\ref{eq:funcadd}) and
with $\mu_k=\mu^*$ for all $k$. One of the
main ideas behind this convergence proof relies on the fact that we can derive
the improvement in $f$ achieved by each step in terms of the change in
$x$. Since the distance between the starting point and the optimal solution,
denoted by $R=\| x_0-x^* \|$, is finite, one can derive an upper bound for the
``accumulative improvement in $f$,''
$\frac{1}{N+1}\sum_{k=0}^N(\mathbb{E}[f(x_k)]-f^*)$.
Hence, we can show that increasing the number of iterations, $N$,  of
Algorithm~\ref{alg:STARS} yields higher accuracy in the solution. 

For simplicity, we denote by $\E[\cdot]$ the expectation over all random
variables
(i.e., $ \E[\cdot] = \E_{u_k, \ldots, u_1,\xi_k, \ldots,
\xi_0}[\cdot]$), unless
otherwise specified. Similarly, we denote $s_{\mu_k} (x_k;u_k,\xi_k,\xi_{k-1}) $
in (\ref{eq:gfo}) by
$s_{\mu_k}$. The following lemma directly follows from
Theorem~\ref{thm.optstep.add}.

\begin{lemma}\label{lemma.s_mu.square}
Let Assumptions \ref{assumption_u}, \ref{assumption_f_add}, and
\ref{assumption_add} hold. 
If the smoothing stepsize $\mu_k$ is set to the constant $\mu^*$ from 
(\ref{eqn.opt.mu}), then Algorithm~\ref{alg:STARS} generates steps satisfying 
\begin{eqnarray*}
\E[\|s_{\mu_k}\|^2] \le 2(n+4)\| \nabla f(x_k) \|^2 +C_2,
\end{eqnarray*}
where $C_2 = 2\sqrt{2}L_1\sigma_a\sqrt{n(n+6)^3}$.
\end{lemma}

\begin{proof} 
Let $g_0(x_k) = \langle \nabla f(x_k), u_k\rangle u_k  $. Then
(\ref{eqn.opt.error.bound.add}) implies that
\begin{eqnarray}\label{eq.s_mu.1}
\E[\|s_{\mu_k}\|^2-2\langle s_{\mu_k},g_0(x_k)\rangle +\|g_0(x_k)\|^2]\le C_1,
\end{eqnarray}
where $C_1=\sqrt{2}L_1\sigma_a \sqrt{n(n+6)^3}$.

The stochastic gradient-free oracle $s_{\mu_k}$ in (\ref{eq:gfo}) is a random
approximation of the gradient $\nabla f(x_k)$. Furthermore, the expectation of
$s_{\mu_k}$ with
respect to $\xi_k$ and $\xi_{k-1}$ yields the forward-difference
approximation of the derivative of $f$ in the direction $u_k$ at $x_k$:
\begin{eqnarray}\label{eq.s_mu.exp}
\E_{\xi_k,\xi_{k-1}}[s_{\mu_k}] = \frac{f(x_k+\mu_k u_k)-f(x_k)}{\mu_k}u_k =
g_\mu(x_k). 
\end{eqnarray}

Combining (\ref{eq.s_mu.1}) and (\ref{eq.s_mu.exp}) yields
\begin{eqnarray}
 \E \left[ \| s_{\mu_k}\|^2\right]&\le&  \E[ 2 \langle s_{\mu_k},g_0(x_k)\rangle - 
\|g_0(x_k)\|^2] + C_1\nonumber\\
&\overset{(\ref{eq.s_mu.exp} )}{=}&   \E_{u_k}[2\langle g_{\mu}(x_k) ,
g_0(x_k)\rangle-  \|g_0(x_k)\|^2]    + C_1\nonumber\\
&=& \E_{u_k}[-\| g_0(x_k) -g_{\mu}(x_k) \|^2  + \| g_{\mu}(x_k) \|^2 ]+ C_1\nonumber\\
&\le & \E_{u_k}[ \| g_{\mu}(x_k) \|^2] + C_1\nonumber\\
&\overset{(\ref{eqn.exp.of.g_mu.square} )}{\le}& 2(n+4)\|\nabla f(x_k)\|+C_2,\nonumber 
\end{eqnarray}
where $C_2 = C_1  +\frac{\mu_k^2}{2}L_1^2(n+6)^3 =2\sqrt{2}L_1\sigma_a
\sqrt{n(n+6)^3}.$
\end{proof}

We are now ready to show convergence of the algorithm. Denote $x^*\in
\mathbb{R}^n$ a
minimizer associated with $f^*=f(x^*)$. Denote by $\mathcal{U}_{k}=
\{u_1,\cdots,u_k  \}$ the set of i.i.d. random variable realizations attached to
each iteration of Algorithm 1. Similarly, let $\mathcal{P}_{k}=\{\xi_0,\cdots,
\xi_k\}$. Define $\phi_0=f(x_0)$ and
$\phi_k=\E_{\mathcal{U}_{k-1},\mathcal{P}_{k-1}}[f(x_k)]$ for $k\ge 1$.

\begin{theorem} \label{convergence_add}
Let Assumptions \ref{assumption_u}, \ref{assumption_f_add}, and
\ref{assumption_add} hold. Let the sequence $\{ x_k \}_{k\ge0}$
be generated by Algorithm 1 with the smoothing stepsize $\mu_k$ set as
$\mu^*$ in 
(\ref{eqn.opt.mu}). If the fixed step length is $h_k = h=\frac{1}{4L_1(n+4)}$
for all $k$, then for any $N\ge 0$, we have
$$\frac{1}{N+1}\sum_{k=0}^N(\phi_k-f^*)\le \frac{4L_1(n+4)}{N+1}\|x_0-x^*\|^2
+\frac{3\sqrt{2}}{5}\sigma_a(n+4).$$
\end{theorem}

\begin{proof}
We start with deriving the expectation of the change in $x$ of each step, that
is,
$\E[r_{k+1}^2]-r_k^2$, where $r_k = \| x_k-x^* \|$. First, 
\begin{eqnarray*}
r_{k+1}^2 & = &\| x_k - h_k s_{\mu_k} -x^* \|^2\nonumber\\
& =& r_k^2-2h_k\langle s_{\mu_k}, x_k-x^*\rangle + h_k^2 \| s_{\mu_k}  \|^2.
\end{eqnarray*}
$\E[s_{\mu_k}]$ can be derived by using (\ref{eqn.exp.of.g_mu}) and
(\ref{eq.s_mu.exp}). $\E[\| s_{\mu_k} \|^2]$ is derived in Lemma
\ref{lemma.s_mu.square}. Hence,
\begin{eqnarray*}
\E\left[ r_{k+1}^2 \right]&\le &  r_k^2 -2h_k \langle \nabla f_\mu(x_k),x_k-x^*
\rangle +h_k^2[2(n+4)\|\nabla f(x_k)\|^2+C_2 ].
 \nonumber
\end{eqnarray*}
By using (\ref{eq:convexity}), (\ref{eq.f_mu.1}), and (\ref{eq:L1third}), we
derive
\begin{eqnarray*}
\E\left[ r_{k+1}^2 \right]&\le &  r_k^2 -2h_k ( f(x_k)-f_\mu(x^*)
)+4h_k^2L_1(n+4)(f(x_k)-f(x^*))+h_k^2C_2 . \nonumber
\end{eqnarray*}
Combining this expression with (\ref{eq.f_mu.2}), which bounds the error
between $f_\mu(x)$ and
$f(x)$, we obtain
\begin{eqnarray*}
\E\left[ r_{k+1}^2 \right]&\le &  r_k^2 -2h_k(1-2h_k L_1(n+4)) ( f(x_k)-f^* )+C_3, \nonumber
\end{eqnarray*}
where $C_3 =h_k^2C_2 +2h_k \frac{\mu_k^2}{2}L_1n=h_k^2C_2 + 
2\sqrt{2}h_k\sigma_a \sqrt{\frac{n^3}{(n+6^3)}}.$

Let $h_k = h=\frac{1}{4L_1(n+4)}$. Then,
\begin{eqnarray}
\E\left[ r_{k+1}^2 \right]&\le &  r_k^2 -\frac{ f(x_k)-f^*}{4L_1(n+4)} +C_3, 
\label{eq.per.step.result}
\end{eqnarray}
where $C_3=\frac{\sqrt{2}\sigma_a}{2L_1}g_1(n)$ and $g_1(n)= \frac{\sqrt{n(n+6)^3}}{4(n+4)^2} + \frac{1}{n+4}\sqrt{\frac{n^3}{(n+6)^3}}$. By showing that $g_1'(n)<0$ for all $n\ge 10$ and $g_1'(n)>0$ for all $n\le 9$, we can prove that $g_1(n)\le \max\{g(9),g(10)\}=\max\{ 0.2936,0.2934 \}\le 0.3$. Hence, $C_3\le \frac{3\sqrt{2}\sigma_a}{20L_1}$.

Taking the expectation in $\mathcal{U}_{k}$ and $\mathcal{P}_{k}$, we have
$$ \E_{\mathcal{U}_{k},\mathcal{P}_{k}} [r^2_{k+1}] \le
\E_{\mathcal{U}_{k-1},\mathcal{P}_{k-1}}[r_k^2 ]-\frac{ \phi_k-f^*}{4L_1(n+4)}
+\frac{3\sqrt{2}\sigma_a}{20L_1}.$$
Summing these inequalities over $k=0,\cdots,N$ and dividing by $N+1$, we obtain
the desired result.
\end{proof}

The bound in Theorem~\ref{convergence_add} is valid also for $\hat{\phi}_N = 
\E_{\mathcal{U}_{k-1},\mathcal{P}_{k-1}} \left[f(\hat{x}_N)\right] $, where 
$\hat{x}_N =\arg \min_x  \{ f(x) : x\in \{ x_0,\cdots,x_N \} \} $. In this case,
\begin{eqnarray}
\E_{\mathcal{U}_{k-1},\mathcal{P}_{k-1}} \left[f(\hat{x}_N)\right] -f^* &\le & \E_{\mathcal{U}_{k-1},\mathcal{P}_{k-1}}   \left[  \dfrac{1}{N+1}\sum_{k=0}^N (\phi_k-f^*)  \right] \nonumber\\
&\le & \dfrac{4L_1(n+4)}{N+1}\| x_0-x^* \|^2+\frac{3\sqrt{2}}{5}\sigma_a(n+4).\nonumber
\end{eqnarray}

Hence, in order to achieve a final accuracy of $\epsilon$ for $\hat{\phi}_N$
(that is, $\hat{\phi}_N-f^*\le \epsilon$), the allowable absolute noise in the
objective function has to satisfy
$\sigma_a\le  \dfrac{5\epsilon}{6\sqrt{2}(n+4)}$. Furthermore, under this bound
on the allowable noise, this $\epsilon$ accuracy can be ensured by {\sf STARS}
in 
\begin{eqnarray}
N=\dfrac{8(n+4)L_1R^2}{\epsilon}-1\sim \mathcal{O}\left(  \dfrac{n}{\epsilon} L_1 R^2\right)
\label{eq:N_add}
\end{eqnarray} 
iterations, where $R^2$ is an upper bound on the squared Euclidean distance
between the starting point and the optimal solution: $\| x_0-x^* \|^2\le
R^2$. 
In other words, given an optimization problem that has bounded absolute noise of
variance $\sigma_a^2$, the best accuracy that can be ensured by {\sf STARS} is 
\begin{eqnarray}
\epsilon_{\mbox{pred}} \ge  \frac{6 \sqrt{2}\sigma_a(n+4)}{5},
\label{eq:predepsadd}
\end{eqnarray}
and we can solve this noisy problem in $\mathcal{O}\left( 
\dfrac{n}{\epsilon_{\mbox{pred}}} L_1 R^2 \right)$ iterations. Unsurprisingly,
a price must be paid for having access only to noisy realizations, and
this price is that arbitrary accuracy cannot be reached in the
noisy setting. 

\section{Multiplicative Noise}\label{sec:mul.noise}

A \textit{multiplicative noise} model is described by
\begin{equation}
\tilde{f}(x;\xi)=f(x)[1+\nu(x;\xi)] = f(x)+f(x)\nu(x;\xi).
\label{eq:funcmul_bounded}
\end{equation}
In practice, $|\nu|$ is bounded by something smaller (often much smaller) than
$1$. A canonical example is when $f$ corresponds to a Monte Carlo integration, 
with the a stopping criterion based on the value $f(x)$.
Similarly, if $f$ is simple and computed in double precision,
the relative errors are roughly $10^{-16}$; in single precision, the errors are
roughly $10^{-8}$ and in half precision we get errors of roughly $10^{-4}$.

Formally,  we make the following assumptions in our analysis of {\sf STARS} for 
the problem (\ref{eqn.prob.chp3}) with multiplicative noise. 

\begin{assumption}[Assumption about $f$]\label{assumption_f_mul} 
$f$ is continuously differentiable and convex and has Lipschitz constant $L_0$.
$\nabla f$ has Lipschitz
constant $L_1$.
\end{assumption}

\begin{assumption}[Assumption about multiplicative
$\nu$]\label{assumption_bounded_noise}
\item[1.] $\nu$ is i.i.d., with zero mean and bounded variance; that is, $\E
[\nu]
=0$, $\sigma_r^2 = \mbox{Var}(\nu)>0$.
\item[2.] The expectation of the signal-to-noise ratio is bounded; that is,
$\E[\frac{1}{1+\nu}]\le b$.
\item[3.] The support of $\nu$ (i.e.,
the range of values that $\nu$ can take with positive probability) is bounded 
by $\pm a$, where $a<1$.
\end{assumption}

The first part of Assumption~\ref{assumption_bounded_noise} is analogous 
to that in Assumption~\ref{assumption_add} and guarantees that the 
distribution of $\nu$ is independent of $x$. Although not specifying a 
distributional form for $\nu$ (with respect to $\xi$), the final two parts of  
Assumption~\ref{assumption_bounded_noise} are made to simplify the presentation 
and rule out cases where the noise completely corrupts the function.


\subsection{Noise and Finite Differences}

Analogous to Theorem \ref{thm.optstep.add}, Theorem \ref{optstepmul_bounded} 
shows how to compute the near-optimal stepsizes in the multiplicative 
noise setting. 

\begin{theorem} \label{optstepmul_bounded}
Let Assumptions~\ref{assumption_f_mul}~and~\ref{assumption_bounded_noise} hold. 
If a forward-difference parameter is chosen as 
\begin{equation}\label{eqn.opt.mu.bounded}
  \mu^* = C_4\sqrt{|f(x) |},\quad \mbox{where } C_4 =\left[ \frac{16\sigma_r^2n }{L_1^2(1+3\sigma_r^2) (n+6)^3}\right]^{\frac{1}{4}},\nonumber
\end{equation}
 then for any $x\in \R^n$ we have 
\begin{equation}\label{eqn.optstepmul_bounded}
\E_{u,\xi_1,\xi_2}[ \mathcal{E}(\mu^*) ] \le 2L_1\sigma_r \sqrt{(1+3\sigma_r^2)n(n+6)^3}|f(x)|+3L_0^2\sigma_r^2(n+4)^2. 
\end{equation}
\end{theorem}

\begin{proof}
By using (\ref{eq:funcmul_bounded}) and (\ref{eq:L1second}), we derive
\begin{eqnarray*}
\mathcal{E}(\mu)&\le& \left\| \frac{f(x+\mu u)\nu(x+\mu u;\xi_1)-f(x)\nu(x;\xi_2)}{\mu} u+ \frac{\mu L_1}{2} \|u\|^2u \right\|^2  \\
&\le& \left( \frac{f(x+\mu u)\nu(x+\mu u;\xi_1)-f(x)\nu(x;\xi_2)}{\mu} + \frac{\mu L_1}{2} \|u\|^2 \right)^2 \|u\|^2 .\end{eqnarray*}
Again applying (\ref{eq:L1second}), we get $\mathcal{E}(\mu)\le X^2\|u\|^2  $,
where
 \begin{eqnarray*}
X&=&\frac{f(x+\mu u)\nu(x+\mu u;\xi_1)-f(x)\nu(x;\xi_2)}{\mu} + \frac{\mu L_1}{2} \|u\|^2\\
&\le&\left(\frac{f(x)}{\mu} +\nabla f(x)^Tu + \frac{\mu L_1}{2} \|u\|^2 \right)\nu(x+\mu u;\xi_1) - \frac{f(x)}{\mu} \nu(x;\xi_2)+ \frac{\mu L_1}{2} \|u\|^2
.\end{eqnarray*}

The expectation of $X$ with respect to $\xi_1$ and $\xi_2$ is 
$$E_{\xi_1,\xi_2}[X]= \frac{\mu L_1}{2} \|u\|^2$$
and the corresponding variance is 
\begin{eqnarray*}
 \mbox{Var} (X)&=& \left(\frac{f(x)}{\mu} +\nabla f(x)^Tu + \frac{\mu L_1}{2} \|u\|^2 \right)^2\sigma_r^2 + \frac{f^2(x)}{\mu^2} \sigma_r^2\\
 &\le& \left(\frac{3f^2(x)}{\mu^2} +3(\nabla f(x)^Tu)^2 + \frac{3\mu^2 L_1^2}{4} \|u\|^4 \right)\sigma_r^2 + \frac{f^2(x)}{\mu^2} \sigma_r^2\\
 &=&\left(\frac{4f^2(x)}{\mu^2} +3(\nabla f(x)^Tu)^2 + \frac{3\mu^2 L_1^2}{4} \|u\|^4 \right)\sigma_r^2 ,
\end{eqnarray*}
where the inequality holds because $(a+b+c)^2\le 3a^2+3b^2+3c^2$ for any $a,b,c$.
Since  $\E[X^2] =\mbox{Var}(X)+(\E[X])^2$, we have that
\begin{eqnarray*}
\E_{\xi_1,\xi_2}[X^2] &\le& \frac{\mu^2 L_1^2(1+3\sigma_r^2)}{4} \|u\|^4+\frac{4\sigma_r^2}{\mu^2} f^2(x)+3(\nabla f(x)^Tu)^2 \sigma_r^2\\
&\le& \frac{\mu^2 L_1^2(1+3\sigma_r^2)}{4} \|u\|^4+\frac{4\sigma_r^2}{\mu^2} 
f^2(x)+3L_0^2 \sigma_r^2 \|u\|^2.
\end{eqnarray*}
Hence, we can derive 
\begin{eqnarray*}
\E[\mathcal{E}(\mu)] &\le& \E_u[  \E_{\xi_1,\xi_2}[   X^2\|u\|^2 ]  ]\\
&=& \E_u[ \|u\|^2 \E_{\xi_1,\xi_2}[   X^2 ]  ]\\
&\le& \E_u\left[  \frac{\mu^2 L_1^2(1+3\sigma_r^2)}{4} 
\|u\|^6+\frac{4\sigma_r^2}{\mu^2} f^2(x)\|u\|^2+3L_0^2 \sigma_r^2 \|u\|^4    
\right].
\end{eqnarray*}

By using (\ref{eq:M_p.le2}), (\ref{eq:M_p.ge2}), and this last expression, we
get
\begin{eqnarray*}
\E[\mathcal{E}(\mu)] &\le&  \frac{\mu^2 L_1^2(1+3\sigma_r^2)}{4} 
(n+6)^3+\frac{4\sigma_r^2n}{\mu^2} f^2(x)+3L_0^2 \sigma_r^2 (n+4)^2  .
\end{eqnarray*}

The right-hand side of this expression is uniformly convex in $\mu$ and 
attains its  global minimum at $\mu^* =C_4\sqrt{|f(x) |}$; the corresponding  
expectation of the least-squares error is 
$$\E_{u,\xi_1,\xi_2}[ \mathcal{E}(\mu^*) ] \le 2L_1\sigma_r \sqrt{(1+3\sigma_r^2)n(n+6)^3}|f(x)|+3L_0^2\sigma_r^2(n+4)^2. $$
\end{proof}

Unlike for the absolute noise case of Section~\ref{sec:additive.noise}, the 
optimal $\mu$ value in Theorem~\ref{optstepmul_bounded} is not 
independent of $x$. Furthermore, letting $\mu_k=\mu^*=C_4\sqrt{|f(x) |}$ 
assumes that $f$ is known. Unfortunately, we have access to $f$ only through 
$\tilde{f}$. However, we can compute an estimate, $\tilde{\mu}$, of $\mu^*$ 
by substituting $f$ with $\tilde{f}$ and still derive an error bound. To
simplify the derivations, we introduce another random variable, 
$\xi_3$, independent of $\xi_1$ and $\xi_2$, to compute $\tilde{\mu}\equiv 
\tilde\mu(x;\xi_3)$. The goal is to obtain an upper bound on 
$\E_{\xi_3}[\E_{\xi_1,\xi_2,u}[\mathcal{E}(\tilde{\mu})]]$, where
$$\mathcal{E}(\tilde\mu)\equiv 
\mathcal{E}(\tilde\mu,x;u,\xi_1,\xi_2,\xi_3)=\left\|  
\frac{\tilde{f}(x+\tilde\mu  ;\xi_1)-\tilde{f}(x;\xi_2)}{\tilde\mu}u-\langle 
\nabla f(x),u\rangle u  \right\|^2.$$ 
This then allows us to proceed with the usual derivations while requiring only 
an additional expectation over $\xi_3$.

\begin{lemma} \label{lemma.tilde.mu.bounded}
Let Assumptions~\ref{assumption_f_mul}~and~\ref{assumption_bounded_noise} hold. 
If a forward-difference parameter is chosen as 
\begin{equation}\label{eqn.tilde.mu.bounded}
  \tilde\mu = C_4\sqrt{|\tilde{f}(x;\xi_3) |},\quad \mbox{where } C_4  =\left[ \frac{16\sigma_r^2n }{L_1^2(1+3\sigma_r^2) (n+6)^3}\right]^{\frac{1}{4}},
\end{equation}
then for any $x\in \R^n$, we have 
\begin{equation}\label{eqn.lemma.tilde.mu.bounded}
\E_{u,\xi_1,\xi_2,\xi_3}[ \mathcal{E}(\tilde{\mu}) ] \le(1+b)L_1\sigma_r \sqrt{(1+3\sigma_r^2)n(n+6)^3}|f(x)|+3L_0^2\sigma_r^2(n+4)^2.
\end{equation}
\end{lemma}

\begin{proof}
\begin{eqnarray}
\E[\mathcal{E}(\tilde{\mu})]& =& \E_{\xi_3}[\E_{u,\xi_1,\xi_2}[\mathcal{E}(\tilde{\mu})]]\nonumber\\
&\le& \E_{\xi_3}\left[ \frac{\tilde\mu^2 L_1^2(1+3\sigma_r^2)}{4} 
(n+6)^3+\frac{4\sigma_r^2n}{\tilde\mu^2} f^2(x)+3L_0^2 \sigma_r^2 (n+4)^2 
\right]\nonumber\\
&=& L_1\sigma_r\sqrt{(1+3\sigma_r^2)n(n+6)^3}| f(x)|\E_{\xi_3}\left[  1+\nu(x;\xi_3)+\frac{1}{1+\nu(x;\xi_3)}  \right]+3L_0^2\sigma_r^2(n+4)^2\nonumber\\
&\le& (1+b)L_1\sigma_r\sqrt{(1+3\sigma_r^2)n(n+6)^3}| f(x)|+3L_0^2\sigma_r^2(n+4)^2,\nonumber
\end{eqnarray}
where the last inequality holds by Assumption~\ref{assumption_bounded_noise}
because the expectation of the signal-to-noise ratio is bounded by $b$.
\end{proof}

\paragraph{Remark:} Similar to the additive noise case, Theorem 
\ref{optstepmul_bounded} and Theorem  \ref{lemma.tilde.mu.bounded} do not 
require $f$ to be convex. Hence, (\ref{eqn.optstepmul_bounded}) and 
(\ref{eqn.lemma.tilde.mu.bounded}) both hold in the nonconvex case. However, 
the following convergence rate analysis applies only to the convex case, since 
Lemma \ref{lemma.s_mu.bound} relies on a convexity assumption for $f$.

\subsection{Convergence Rate Analysis}

Let $\mu_k=\tilde{\mu}=C_4\sqrt{|\tilde{f}(x_k;\xi_{k'}) |}$ in 
Algorithm~\ref{alg:STARS}. Before showing the convergence result, we derive 
$\E[\langle s_{\tilde{\mu}},x_k-x^*   \rangle]$ and $\E[\|s_{\tilde{\mu}}\|^2]$, 
where $s_{\tilde{\mu}}$ denotes $  s_\mu(x_k;u_k,\xi_k,\xi_{k-1},\xi_{k'})$ and 
$\E[\cdot] $ denotes the expectation over all random variables $u_k,\xi_k, 
\xi_{k-1}$, and $\xi_{k'}$(i.e.,  $\E[\cdot] = 
\E_{u_k,\xi_k,\xi_{k-1},\xi_{k'}}[\cdot]$), unless otherwise specified. 

\begin{lemma}\label{lemma.s_mu.square.bound}
Let Assumptions~\ref{assumption_f_mul}~and~\ref{assumption_bounded_noise} hold. 
 If $\mu_k=\tilde{\mu}=C_4\sqrt{|\tilde{f}(x_k;\xi_{k'}) |}$, then
\begin{eqnarray*}
\E[\| s_{\tilde{\mu}}\|^2] \le 2(n+4)\| \nabla f(x_k) \|^2 +C_5|f(x_k)|+C_6,
\end{eqnarray*}
where $C_5 =\frac{1}{2}C_4^2L_1^2(n+6)^3+ 
(1+b)L_1\sigma_r\sqrt{(1+3\sigma_r^2)n(n+6)^3}$ and 
$C_6=3L_0^2\sigma_r^2(n+4)^2$.
\end{lemma}

\begin{proof} Let $g_0(x_k)=\langle \nabla f(x_k),u_k \rangle u_k .$ The bound 
(\ref{eqn.tilde.mu.bounded}) in Theorem~\ref{lemma.tilde.mu.bounded} implies 
that
\begin{eqnarray*}
\E [\| s_{\tilde\mu}- g_0(x_k)\|^2 ]  &\le &  (1+b)L_1\sigma_r\sqrt{(1+3\sigma_r^2)n(n+6)^3}| f(x)|+3L_0^2\sigma_r^2(n+4)^2 \equiv \ell(x).
\end{eqnarray*}
Hence, 
\begin{eqnarray*}
&& \E \left[ \| s_{\tilde\mu}\|^2\right]\\
&\le& \E_{\xi_{k'}} \left[ \E _{u_k,\xi_k,\xi_{k-1}}[ 2 \langle s_{\mu},g_0(x_k)\rangle -  \|g_0(x_k)\|^2]  \right] +\ell(x)\nonumber\\
&\overset{(\ref{eq.s_mu.exp} )}{=}& \E_{\xi_{k'}} \left[ \E _{u_k}[ 2 \langle g_{\mu_k}(x_k),g_0(x_k)\rangle -  \|g_0(x_k)\|^2]    \right] + \ell(x)  \nonumber\\
&\le & \E_{\xi_{k'}} \left[\E_{u_k}[ \| g_{\mu_k}(x_k) \|^2] \right] + \ell(x)\nonumber\\
&\overset{(\ref{eqn.exp.of.g_mu.square} )}{\le}&2(n+4)\| \nabla f(x_k) \|^2 +\E_{\xi_{k'}} \left[\frac{\mu_k^2}{2}L_1^2(n+6^3)\right] + \ell(x)\\
&=& 2(n+4)\| \nabla f(x_k) \|^2 +C_5|f(x_k)|+C_6, \nonumber 
\end{eqnarray*}
where the last equality holds since $\E_{\xi_{k'}} [\mu_k^2]=\E_{\xi_{k'}} [C_4^2|f(x_k)| (1+\nu(x_k;\xi_{k'}) ]=C_4^2|f(x_k)| .$
\end{proof}

\begin{lemma}\label{lemma.s_mu.bound}
Let Assumptions~\ref{assumption_f_mul}~and~\ref{assumption_bounded_noise} hold. 
If $\mu_k=\tilde{\mu}=C_4\sqrt{|\tilde{f}(x_k;\xi_{k'}) |}$, then
\begin{eqnarray*}
\E[\langle s_{\tilde{\mu}},x_k-x^*   \rangle] \ge f(x_k)-f^*-\frac{C_4^2L_1n}{2}|f(x_k)|.
\end{eqnarray*}
\end{lemma}

\begin{proof} First, we have
\begin{eqnarray*}
\E_{u_k,\xi_k,\xi_{k-1}}[s_{\tilde{\mu}}]&=&\E_{u_k,\xi_k,\xi_{k-1}} 
\left[\frac{\tilde{f}(x_k+\mu_ku_k;\xi_k)-\tilde{f}(x_k;\xi_{k-1})}{\mu_k}u_k 
\right]\\
&=&\E_{u_k,\xi_k,\xi_{k-1}} 
\left[\frac{f(x_k+\mu_ku_k)[1+\nu(x_k+\mu_ku_k;\xi_k)]-f(x_k)[1+\nu(x_k;\xi_{k-1
})]}{\mu_k}u_k \right]\\
&=&\E_{u_k}\left[ \frac{f(x_k+\mu_ku_k)-f(x_k)}{\mu_k}u_k \right]\\
&=&\E_{u_k}[g_{{\mu}_k}(x_k)]\\
&\overset{(\ref{eqn.exp.of.g_mu} )}{=}&\nabla f_{\mu_k}(x_k).
\end{eqnarray*}
Then, we get
\begin{eqnarray*}
\E_{u_k,\xi_k,\xi_{k-1}}[\langle s_{\tilde{\mu}},x_k-x^* \rangle]&=& \langle 
\nabla f_{\mu_k}(x_k),x_k-x^*   \rangle\\
&\overset{(\ref{eq:convexity} )}{\ge}&f_{\mu_k}(x_k)-f_{\mu_k}(x^*)  \\
&\overset{(\ref{eq.f_mu.1} )}{\ge}&f(x_k)-f_{\mu_k}(x^*)\\
&\overset{(\ref{eq.f_mu.2} )}{\ge}& f(x_k)-f^*-\frac{\mu_k}{2}L_1n.
\end{eqnarray*}
Since $\mu_k=\tilde{\mu}=C_4\sqrt{|\tilde{f}(x_k;\xi_{k'}) |}$, we have
$$\E[\langle s_{\tilde{\mu}},x_k-x^*  \rangle] = \E_{\xi_{k'}}[ 
\E_{u_k,\xi_k,\xi_{k-1}}[\langle s_{\tilde{\mu}},x_k-x^*   \rangle]]  \ge 
f(x_k)-f^*-\frac{C_4^2L_1n}{2}|f(x_k)|.$$
\end{proof}

We are now ready to show the convergence of Algorithm~\ref{alg:STARS}, with 
$\mu_k=\tilde{\mu}$, for the minimization of a function  
(\ref{eq:funcmul_bounded}) with bounded multiplicative noise.

\begin{theorem}\label{algthmmul}
Let Assumptions~\ref{assumption_f_mul}~and~\ref{assumption_bounded_noise} hold. 
Let the sequence 
$\{x_k\}_{k\ge 0}$ be generated by Algorithm~\ref{alg:STARS} with the smoothing 
parameter $\mu_k$ being
$$ \mu_k  = \tilde\mu = C_4\sqrt{|\tilde{f}(x;\xi_{k'}) |}$$
and the fixed step length set to $h_k = h=\frac{1}{4L_1(n+4)}$ for all $k$. Let 
$M$ be an upper bound on the average of the historical absolute values of 
noise-free function evaluations; that is,
$$M \ge \dfrac{1}{N+1} \sum_{k=0}^N | \phi_k  | = \dfrac{1}{N+1} \left( | f(x_0) 
| +\sum_{k=1}^N \E_{\mathcal{U}_{k-1},\mathcal{P}_{k-1}} \left[  | f(x_k)   |  
\right] \right).$$
Then, for any $N\ge 0$ we have
\begin{eqnarray}
\dfrac{1}{N+1}\sum_{k=0}^N (\phi_k-f^*)\le \dfrac{4L_1(n+4)}{N+1}\| x_0-x^* \|^2+
4L_1(n+4)\left( C_7M+C_8  \right),
\label{eq:resultmul}
\end{eqnarray}
where $C_7 = \frac{C_4^2n}{4(n+4)}+\frac{C_5}{16L_1^2(n+4)^2}$ and
$C_8=\frac{C_6}{16L_1^2(n+4)^2}$.
\end{theorem}

\begin{proof}

Let $r_k = \| x_k-x^* \|$. First, 
\begin{eqnarray*}
r_{k+1}^2 & = &\| x_k - h_k s_{\tilde\mu} -x^* \|^2\nonumber\\
& =& r_k^2-2h_k\langle s_{\tilde\mu}, x_k-x^*\rangle + h_k^2\|s_{\tilde\mu}\|^2.
\end{eqnarray*}
$\E[\langle s_{\tilde\mu}, x_k-x^*\rangle]$ and $\E[\| s_{\tilde\mu} \|^2]$ are 
derived in Lemma~\ref{lemma.s_mu.bound} and Lemma \ref{lemma.s_mu.square.bound}, 
respectively. Hence, incorporating (\ref{eq:L1third}), we derive
\begin{eqnarray*}
\E\left[ r_{k+1}^2 \right]&\le &  r_k^2 -2h_k  ( 
f(x_k)-f^*-\frac{C_4^2L_1n}{2}|f(x_k)| ) +h_k^2[   2(n+4)\| \nabla f(x_k) \|^2 
+C_5|f(x_k)|+C_6   ] \nonumber\\
&\le&r_k^2-2h_k(1-2h_kL_1(n+4))(f(x_k)-f^*) +(h_kC_4^2L_1n+h_k^2C_5)|f(x_k)|+h_k^2C_6.
\end{eqnarray*}

Let $h_k=\frac{1}{4L_1(n+4)}$. Then, taking the expectation with respect to 
$\mathcal{U}_{k}=\{ u_1,\cdots, u_k \}$ and 
$\mathcal{P}_k=\{\xi_0,\xi_0',\xi_1,\xi_{1'},\cdots,\xi_{k}\}$ yields
\begin{eqnarray}
\E_{\mathcal{U}_{k},\mathcal{P}_{k}}\left[ r_{k+1}^2\right]&\le&  
\E_{\mathcal{U}_{k-1},\mathcal{P}_{k-1}}\left[r_{k}^2\right] -\dfrac{ \phi_k-f^* 
}{ 4L_1(n+4) } + C_7 | \phi_k|+C_8.\nonumber
\end{eqnarray}

Summing these inequalities over $k=0,\cdots,N$ and dividing by $N+1$, we get 
\begin{eqnarray}
\dfrac{1}{N+1}\sum_{k=0}^N (\phi_k-f^*)&\le& \dfrac{4L_1(n+4)}{N+1}\| x_0-x^* \|^2+
4L_1(n+4)(C_7M+C_8).\nonumber
\end{eqnarray}
\end{proof}

The bound (\ref{eq:resultmul}) is valid also for $\hat{\phi}_N = 
\E_{\mathcal{U}_{k-1},\mathcal{P}_{k-1}} \left[f(\hat{x}_N)\right] $, where  
$\hat{x}_N =\arg \min_x  \{ f(x) : x\in \{ x_0,\cdots,x_N \} \} $. In this case,
\begin{eqnarray}\label{eq.rel.final.result}
\E_{\mathcal{U}_{k-1},\mathcal{P}_{k-1}} \left[f(\hat{x}_N)\right] -f^* &\le & 
\E_{\mathcal{U}_{k-1},\mathcal{P}_{k-1}}   \left[  \dfrac{1}{N+1}\sum_{k=0}^N 
(\phi_k-f^*)  \right] \nonumber\\
&\le &\dfrac{4L_1(n+4)}{N+1}\| x_0-x^* \|^2+
4L_1(n+4)(C_7M+C_8).
\end{eqnarray}

Let us collect and simplify the constants $C_7$ and $C_8$. First,
$C_8=\frac{C_6}{16L_1^2(n+4)^2}=\frac{3L_0^2\sigma_r^2}{16L_1^2}$. Second,
since  
\begin{eqnarray*}
C_5 & = &\frac{1}{2}C_4^2L_1^2(n+6)^3+ (1+b)L_1\sigma_r\sqrt{(1+3\sigma_r^2)n(n+6)^3}\\
&=& 2L_1\sigma_r\sqrt{\frac{1}{1+3\sigma_r^2}}\sqrt{n(n+6)^3}+ (1+b)L_1\sigma_r\sqrt{(1+3\sigma_r^2)n(n+6)^3}\\
&\le& (b+3)L_1\sigma_r\sqrt{1+3\sigma_r^2}\sqrt{n(n+6)^3},
\end{eqnarray*}
where the last inequality holds because $\frac{1}{1+3\sigma_r^2}\le1\le
1+3\sigma_r^2$, we can derive
\begin{eqnarray*}
C_7 &=& \frac{C_4^2n}{4(n+4)}+\frac{C_5}{16L_1^2(n+4)^2}\\
&\le& \frac{1}{L_1}\sqrt{\frac{\sigma_r^2}{1+3\sigma_r^2}}\cdot \frac{n}{n+4}\sqrt{\frac{n}{(n+6)^3}}
+ \frac{(b+3)\sigma_r\sqrt{1+3\sigma_r^2}}{16L_1} \cdot \frac{\sqrt{n(n+6)^3}}{(n+4)^2}\\
&\le & \frac{\sigma_r\sqrt{1+3\sigma_r^2}}{L_1}\left[ g_2(n)  +  (b+3) g_3(n)
\right],
\end{eqnarray*}
where $g_2(n)=\frac{n}{n+4}\sqrt{\frac{n}{(n+6)^3}}$,
$g_3=\frac{\sqrt{n(n+6)^3}}{16(n+4)^2}$, and the last inequality again utilizes
$\frac{1}{1+3\sigma_r^2}\le1\le 1+3\sigma_r^2$. It can be shown that $g_2'(n)<0$
for all $n\ge 8$ and $g_2'(n)>0$ for all $n\le 7$, thus $g_2(n)\le
\max\{g(7),g(8)\}=\max\{ 0.0359,0.0360 \}\le \frac{3}{64}$. Similarly, one can
prove that $g_3'(12)=0$, $g_3'(n)<0$ for all $n> 12$, and $g_3'(n)>0$ for all
$n< 12$, which indicates $g_3(n)\le g_3(12)\approx 0.0646\le \frac{3}{32}$.
Hence,
\begin{eqnarray*}
C_7 &\le& \frac{3(2b+7)\sigma_r\sqrt{1+3\sigma_r^2}}{64L_1} \le
\frac{3\sqrt{3}(2b+7)(\sigma_r^2+\frac{1}{6})}{64L_1},
\end{eqnarray*}
where the last inequality holds because
$\sigma_r\sqrt{\frac{1}{3}+\sigma_r^2}\le\sigma_r^2+\frac{1}{6} $.

With $C_7$ and $C_8$ simplified, (\ref{eq.rel.final.result}) can be used to
establish an accuracy $\epsilon$ for $\hat{\phi}_N$; that is,
$\hat{\phi}_N-f^*\le 
\epsilon$, can be achieved in $\mathcal{O}\left(  \dfrac{n}{\epsilon} L_1 
R^2\right)$ iterations, provided the variance of the relative noise
$\sigma_r^2$ 
satisfies
\begin{eqnarray*}\label{eq:predepsrel}  
4L_1(n+4)(C_7M+C_8)\le
\frac{1}{2}C_9(\sigma_r^2+\frac{1}{6})(n+4)\le\frac{\epsilon}{2},    
\end{eqnarray*} 
where $ C_9 = \frac{3\sqrt{3}}{8} (2b+7)M+\frac{3L_0^2}{2L_1}$, that is,
\begin{eqnarray}\label{eq:predepsrel}  
\sigma_r^2 \le\frac{\epsilon}{C_9 (n+4)}-\frac{1}{6}.
\end{eqnarray} 

The bound in (\ref{eq:predepsrel}) may be cause for concern since the upper
bound may only be positive for larger values of $\epsilon$. Rearranging the
terms explicitly shows that the additive term $\frac{1}{6}$ is a limiting factor
for the best accuracy that can be ensured by this bound: 
\begin{eqnarray}
\epsilon_{\mbox{pred}}\ge C_9(\sigma_r^2+\frac{1}{6}) (n+4).
\end{eqnarray}

\section{Numerical Experiments}
\label{sec:experiments}

We perform three types of numerical studies. Since our convergence rate
analysis guarantees only that the means converge, we first test how much
variability the performance of {\sf  STARS} show from one run to another.
Second, we study the convergence behavior of {\sf STARS} in both the absolute 
noise and multiplicative noise cases and examine these results relative to the
bounds established in our analysis. Then, we compare {\sf STARS} with four 
other randomized zero-order methods to highlight what is gained by using an
adaptive smoothing stepsize. 

\subsection{Performance Variability}\label{sec.noise.invariant}
We first examine the variability of the performance of {\sf STARS} relative to 
that of Nesterov's {\sf RG} 
algorithm \cite{Nest2011}, which is summarized in Algorithm \ref{alg:RG}. One
can observe 
that {\sf RG} and {\sf STARS} have identical algorithmic updates except for the 
choice of the smoothing stepsize $\mu_k$. Whereas {\sf STARS} takes into
account 
the noise level, {\sf RG} calculates the smoothing stepsize based on the target 
accuracy $\epsilon$ in addition to the problem dimension and Lipschitz
constant,
 \begin{equation}
 \label{eqn.rg.smoothing.stepsize}
 \mu  = \frac{5}{3(n+4)} \sqrt{\frac{\epsilon}{2L_1}}.
 \end{equation}

%

\balgorithm[t]
\caption{({\sf RG}: Random Search for Smooth Optimization)}
\label{alg:RG}
 \balgorithmic[1]
\STATE Choose initial point $x_0$ and iteration limit $N$. Fix step length 
$h_k = h = \frac{1}{4(n+4)L_1 }$ and compute smoothing stepsize $ \mu_k$ based 
on $\epsilon=2^{-16}$. Set $k \gets 1$.\\
\STATE Generate a random Gaussian vector $u_k$. \\
\STATE Evaluate the function values $\tilde{f}(x_k;\xi_k)$ and $\tilde{f}(x_k + 
\mu_ku_k;\xi_k)$.\\
\STATE Call the random stochastic gradient-free oracle \\
$$s_\mu (x_k;u_k,\xi_k) = 
\frac{\tilde{f}(x_k+\mu_ku_k;\xi_k)-\tilde{f}(x_k;\xi_{k})}{\mu_k}u_k.$$
\STATE Set  $x_{k+1}=x_k-h_k s_\mu (x_k;u_k,\xi_k)$, update $k\gets k+1$, and 
return to Step 2.
  \ealgorithmic
\ealgorithm

\begin{figure}[ht!]
     \begin{center}
        \subfigure[$\sigma_a = 10^{-6}$]{%
            \label{fig:shade.abs.6}
            \includegraphics[width=0.5\textwidth]{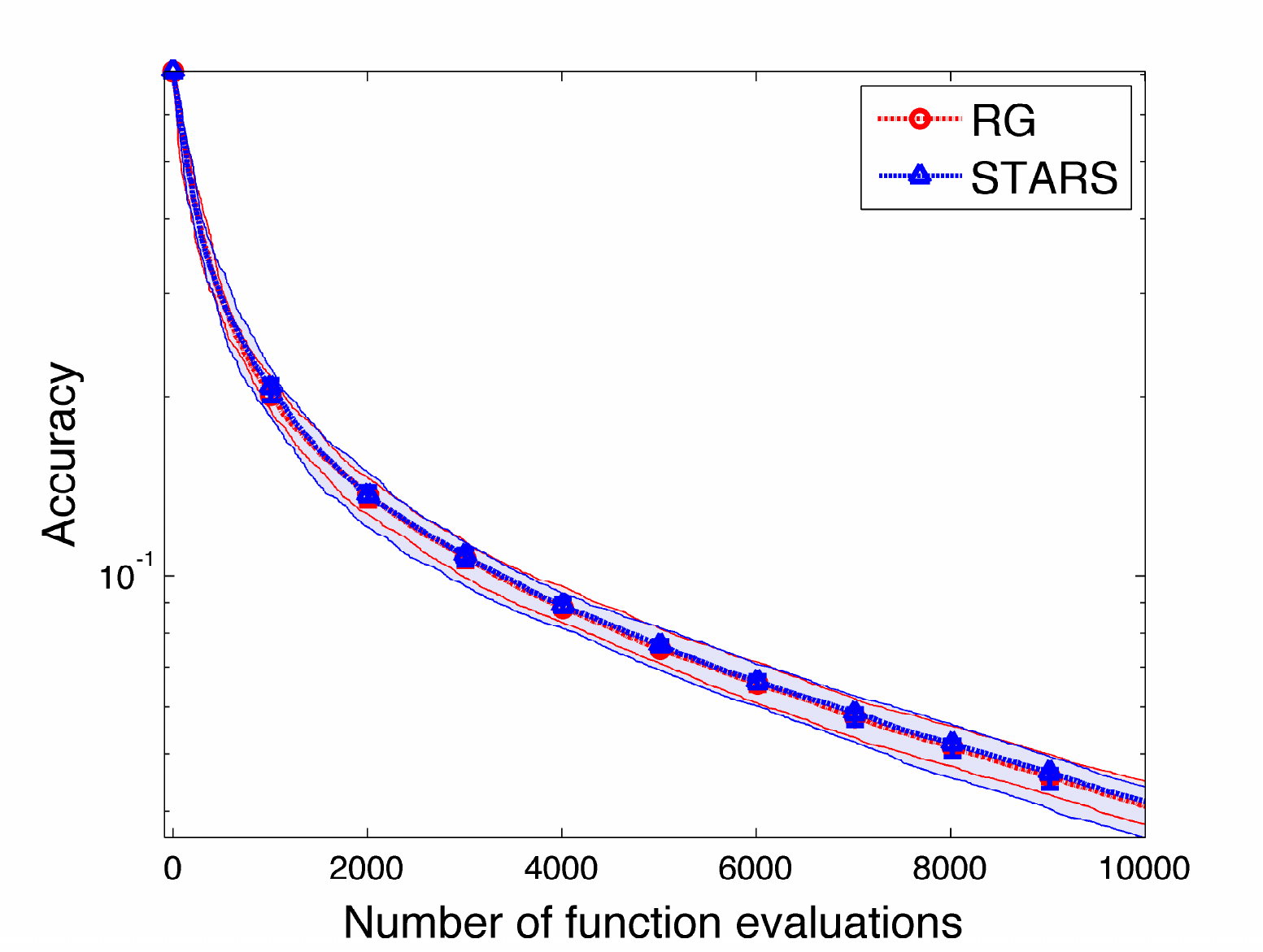}
        }%
        \subfigure[$\sigma_a = 10^{-3}$]{%
           \label{fig:shade.abs.4}
           \includegraphics[width=0.5\textwidth]{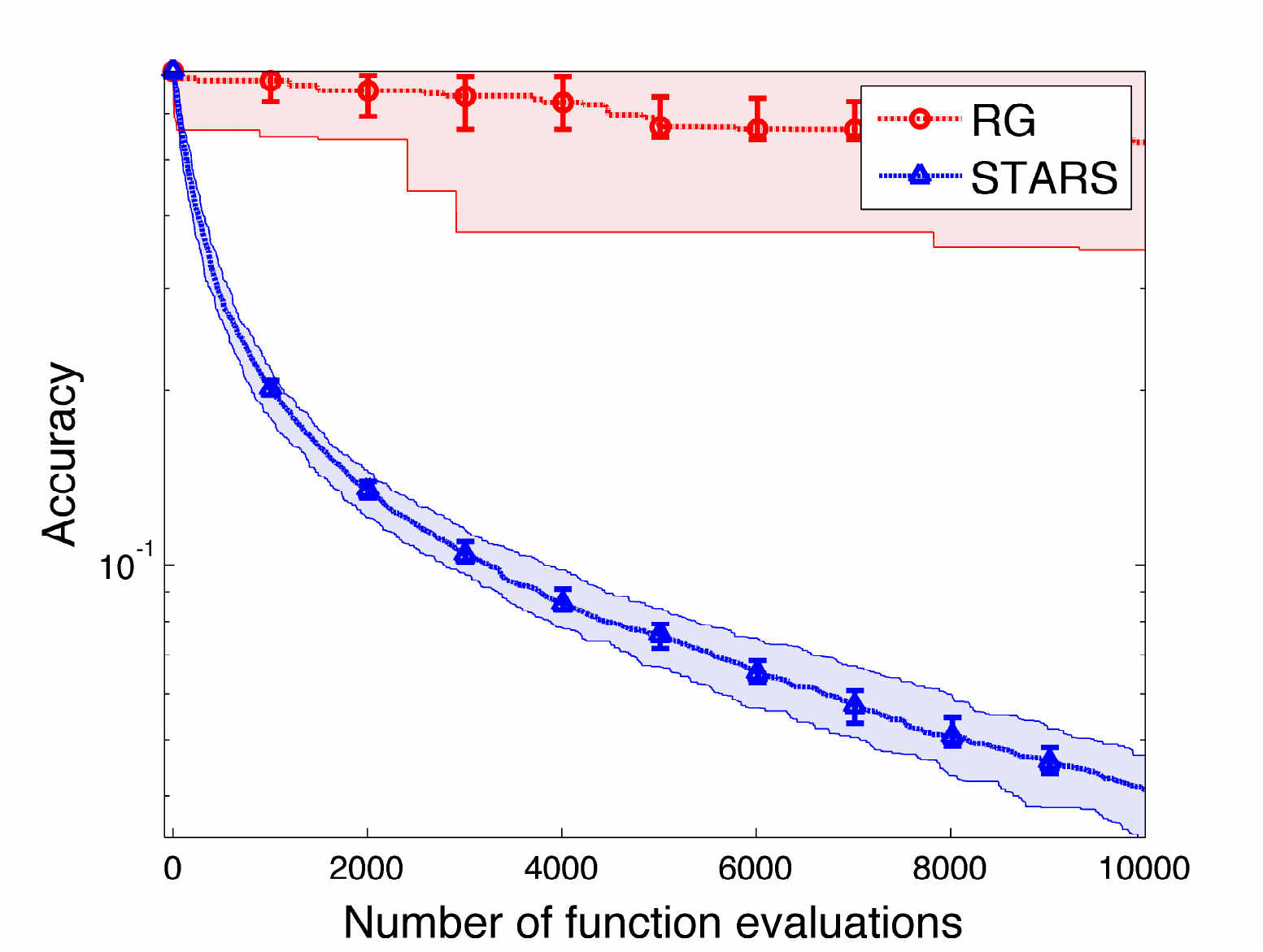}
        }\\
                \subfigure[$\sigma_r = 10^{-6}$]{%
            \label{fig:shade.rel.6}
            \includegraphics[width=0.5\textwidth]{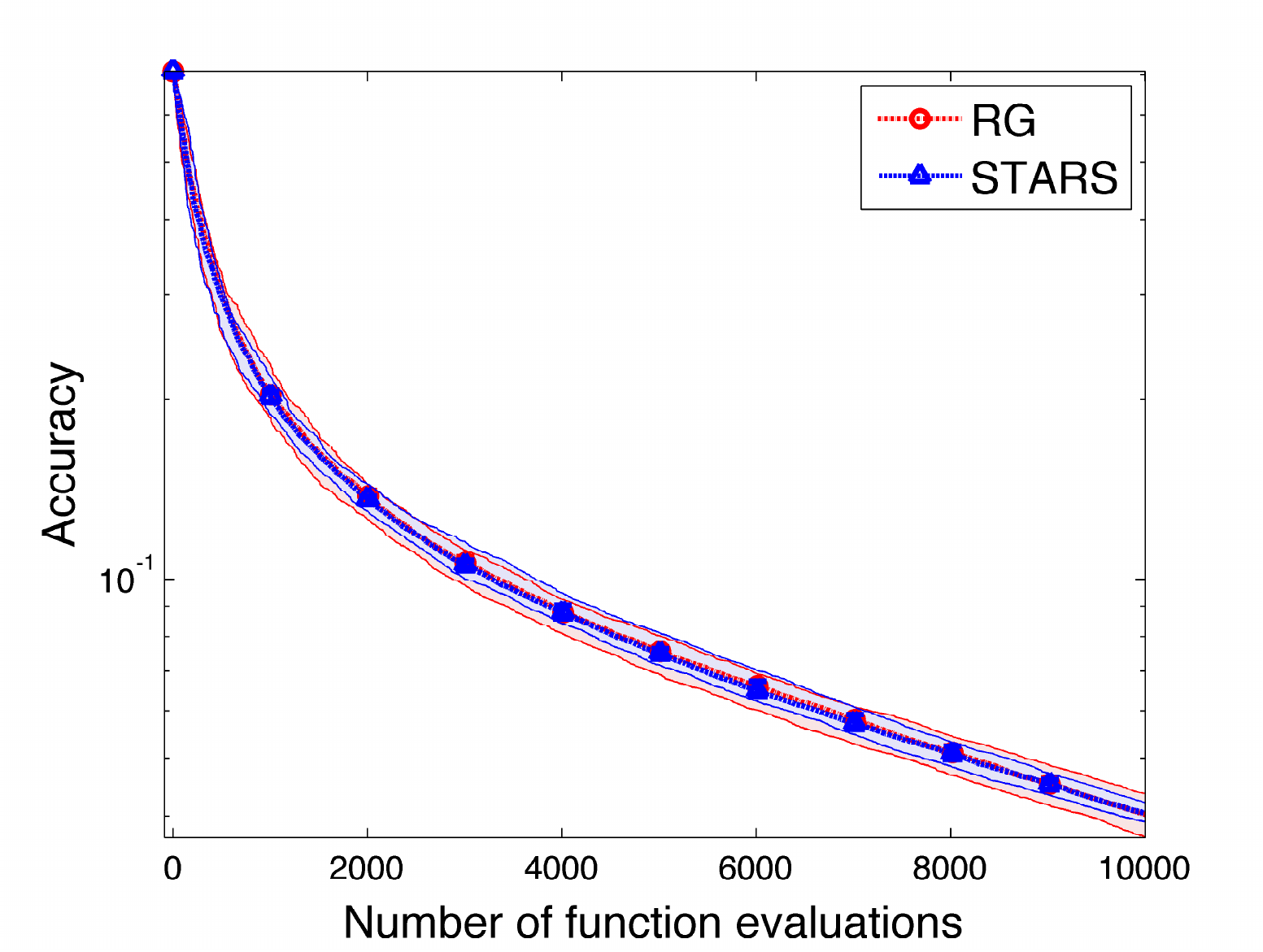}
        }%
        \subfigure[$\sigma_r = 10^{-3}$]{%
           \label{fig:shade.rel.3}
           \includegraphics[width=0.5\textwidth]{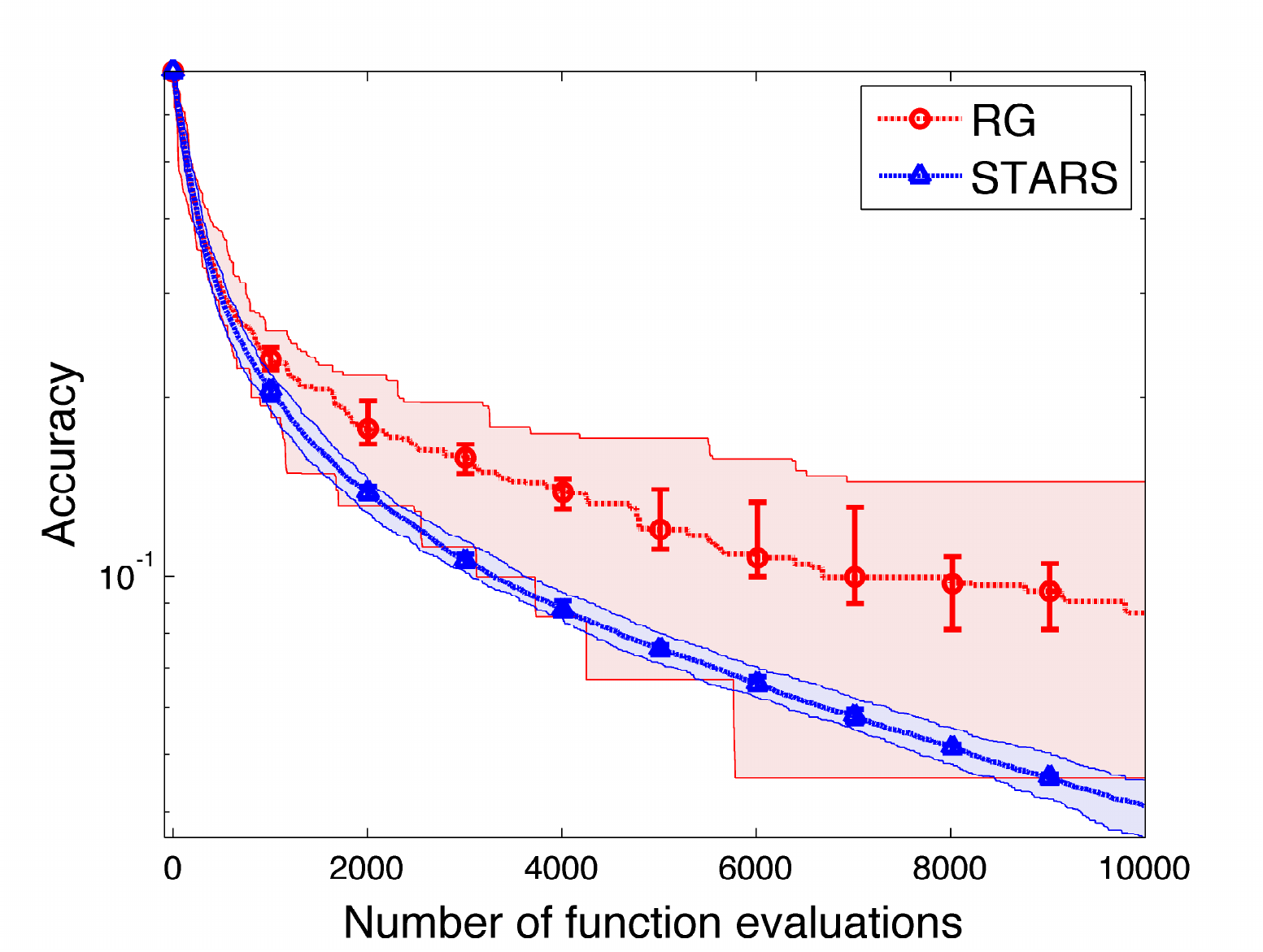}
        }
            \end{center}
    \caption{Median and quartile plots of achieved accuracy with respect to 20 
random seeds when applying {\sf RG} and {\sf STARS} to the noisy $f_1$ 
function. Figures \ref{sec:experiments}.\ref{fig:shade.abs.6} and 
\ref{sec:experiments}.\ref{fig:shade.abs.4} show the additive noise case, while 
Figures \ref{sec:experiments}.\ref{fig:shade.rel.6} and 
\ref{sec:experiments}.\ref{fig:shade.rel.3} show the multiplicative noise case.}
   \label{fig:shade}
\end{figure}

$\mathsf{MATLAB}$ implementations of both {\sf RG} and {\sf STARS} are tested on 
a smooth convex function with random noise added in both additive and 
multiplicative forms. 
In our tests, we use uniform random noise, with $\nu$ generated uniformly from
the interval $[-\sqrt{3}\sigma,\sqrt{3}\sigma]$ by using $\mathsf{MATLAB}$'s
random number generator $\mathsf{rand}$. This choice ensures that $\nu$ has zero
mean and bounded variance $\sigma^2$ in both the 
additive ($\sigma_a = \sigma$) and multiplicative cases ($\sigma_r = \sigma$) 
and that Assumptions~\ref{assumption_add}~and~\ref{assumption_bounded_noise}
hold, provided that $\sigma<3^{-1/2}$.

We use Nesterov's smooth 
function as introduced in \cite{Nest2011}: 
\begin{equation}\label{eqn.smooth.function}
f_1(x) = \dfrac{1}{2}(x^{(1)} )^2 + \dfrac{1}{2} \sum_{i=1}^{n-1}  (x^{(i+1)} 
-x^{i} )^2  +\dfrac{1}{2}(x^{(n)} )^2 -x^{(1)},
\end{equation}
where $x^{(i)}$ denotes the $i$th component of the vector $x\in \R^n$. The 
starting point specified for this problem is the vector of zeros,  $x_0 = 
\textbf{0}$. 
The optimal solution is $$ x^{*(i)} = 1-\dfrac{i}{n+1},\, i = 1,\cdots, n;\quad 
f(x^*) =- \dfrac{n}{2(n+1)} .$$

The analytical values for the parameters (corresponding to Lipschitz constant 
for the gradient and the squared Euclidean 
distance between the starting point and optimal solution)
are: $L_1\le 4$ and $R^2 =\| x_0-x^* \|^2 \le \dfrac{n+1}{3}.$ Both methods 
were given the same parameter value (4.0) for $L_1$, but the smoothing 
stepsizes differ. Whereas {\sf RG} always uses fixed stepsizes of the form 
(\ref{eqn.rg.smoothing.stepsize}), {\sf STARS} uses fixed stepsizes of the form 
(\ref{eqn.opt.mu}) in the absolute noise case and uses dynamic stepsizes 
calculated as (\ref{eqn.tilde.mu.bounded}) in the multiplicative noise case. To
observe convergence over many random trials, we use a small problem 
dimension of $n=8$; however, the behavior shown in Figure~\ref{fig:shade} 
is typical of the behavior that we observed in higher dimensions 
(but the $n=8$ case requiring fewer function evaluations).

In Figure \ref{fig:shade}, we plot the accuracy achieved at each function 
evaluation, which is the true function value $f(x_k)$ minus the optimal 
function value $f(x^*)$. The median across 20 trials is plotted as a line; 
the shaded region denotes the best and worst trials; and the 25\% and 75\% 
quartiles are plotted as error bars. We observe that when the function 
is relatively smooth, as in Figure~\ref{sec:experiments}.\ref{fig:shade.abs.6} 
when the additive noise is $10^{-6}$, the methods exhibit similar performance. 
As the function gets more noisy, however, as in Figure 
\ref{sec:experiments}.\ref{fig:shade.abs.4} when the additive noise becomes 
$10^{-4}$, {\sf RG} shows more fluctuations in performance resulting in large 
variance, whereas the performance {\sf STARS} is almost the same as in the 
smoother case. The same noise-invariant behavior of {\sf STARS} can be 
observed in the multiplicative case.  

\subsection{Convergence Behavior}

\begin{figure}[t!]
     \begin{center}
        \subfigure[$\sigma_a = 10^{-2}$]{%
            \label{fig:convergence.abs.2}
            \includegraphics[height=2.3in]{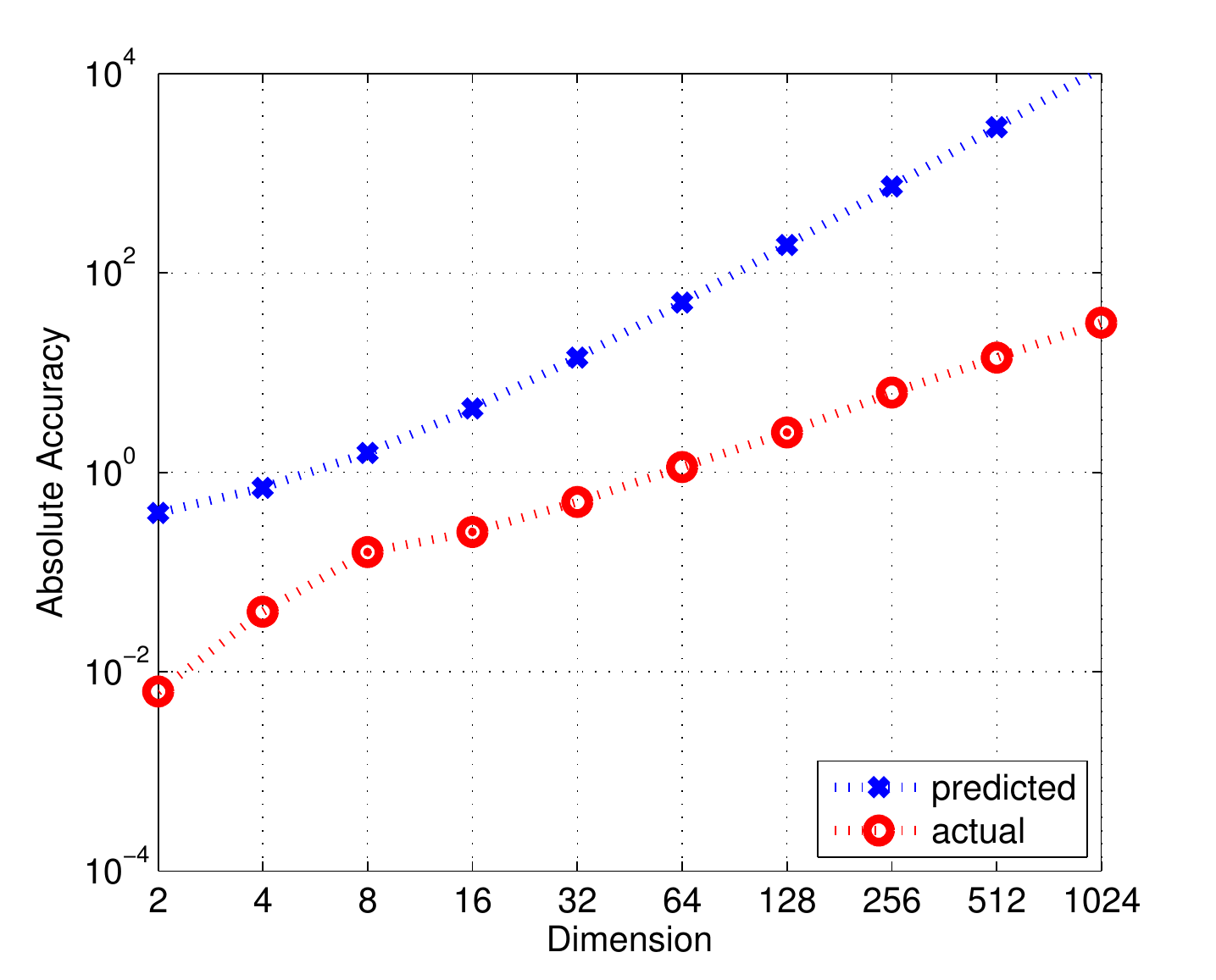}
        }%
        \subfigure[$\sigma_a = 10^{-4}$]{%
           \label{fig:convergence.abs.4}
           \includegraphics[height=2.3in]{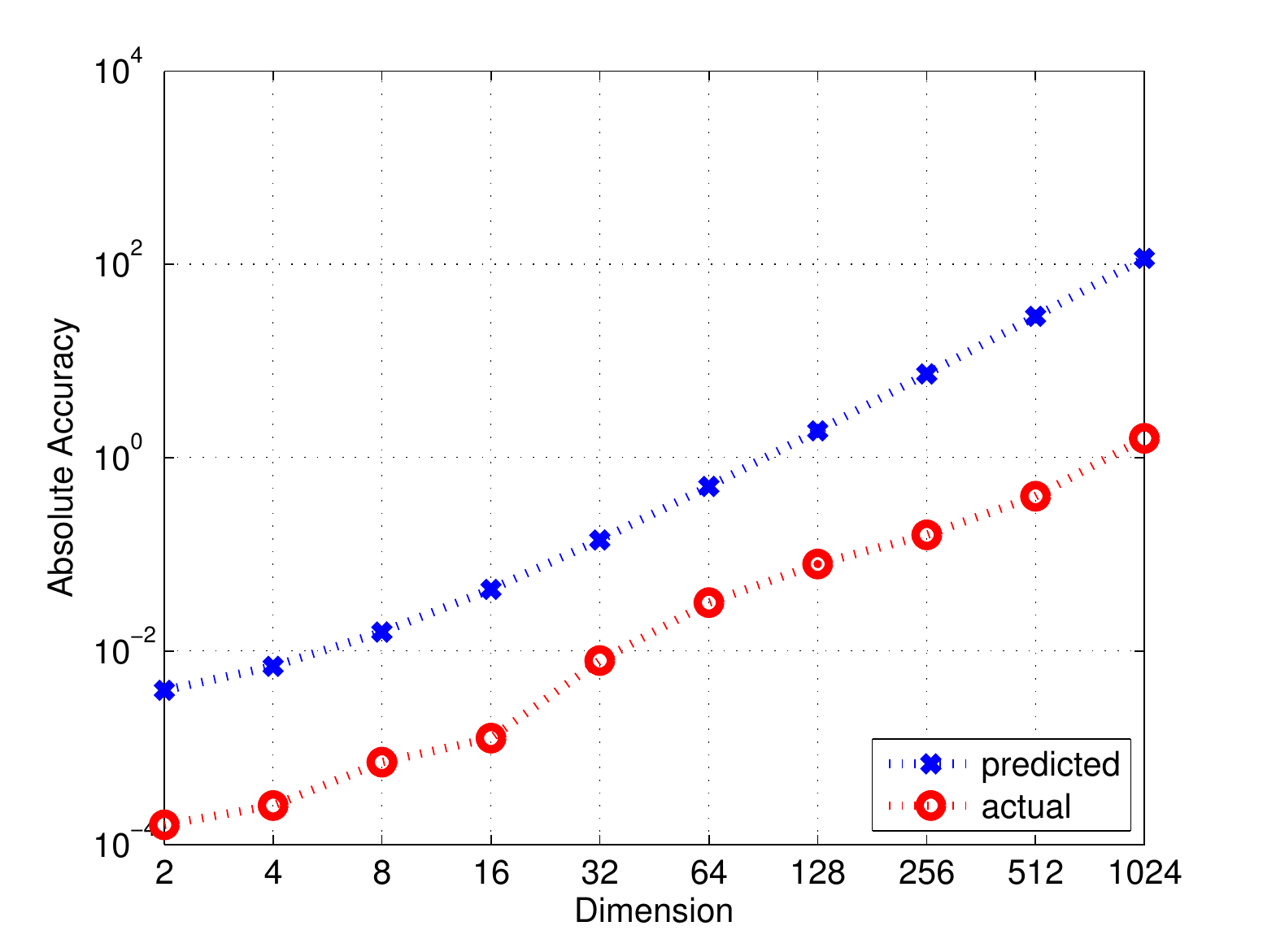}
        }\\
	\subfigure[$\sigma_r = 10^{-4}$]{%
            \label{fig:convergence.rel.4}
            \includegraphics[height=2.3in]{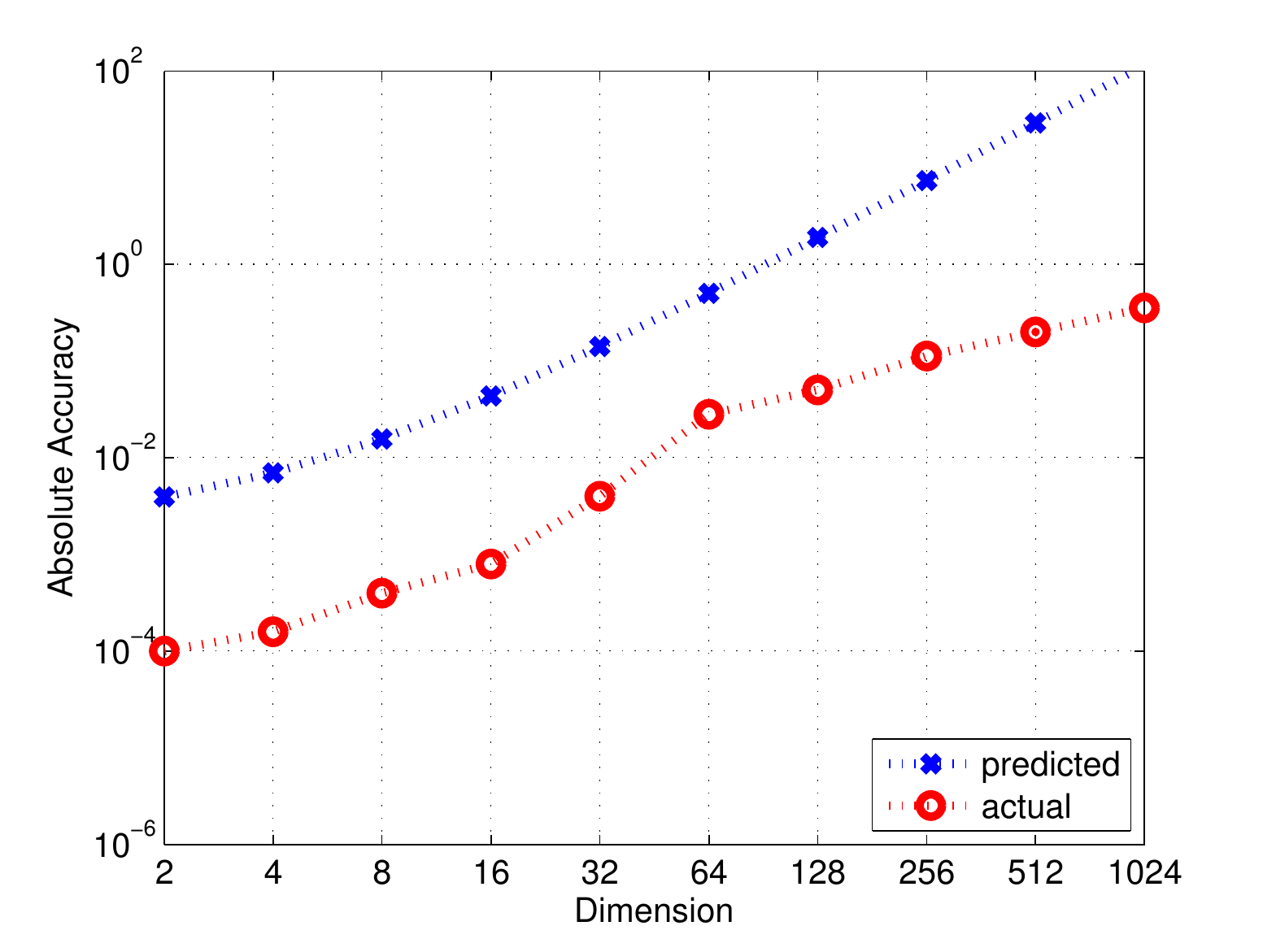}
        }%
        \subfigure[$\sigma_r = 10^{-6}$]{%
           \label{fig:convergence.rel.6}
           \includegraphics[height=2.3in]{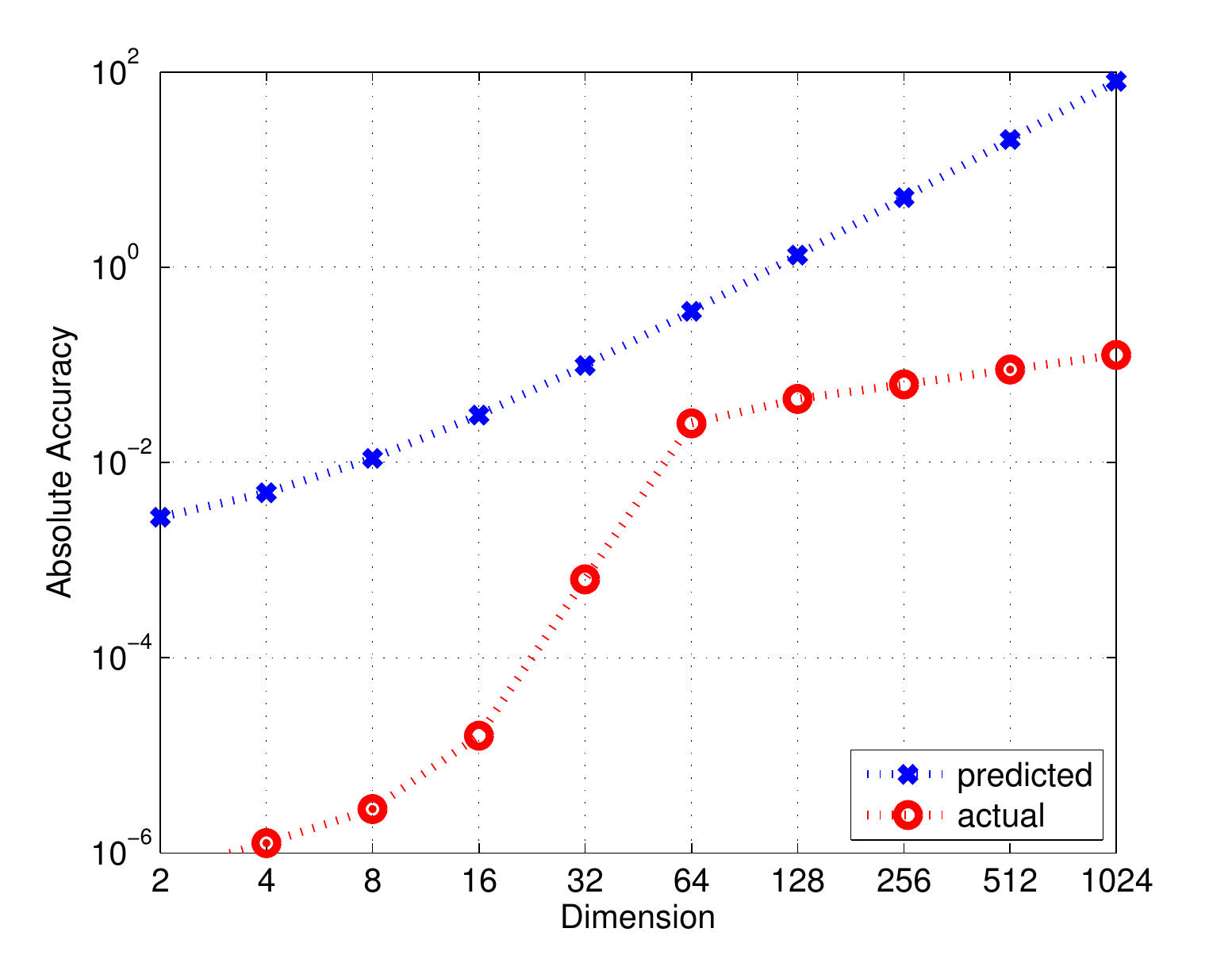}
        }
            \end{center}
    \caption{Convergence behavior of {\sf STARS}: absolute accuracy versus 
dimension $n$. Two absolute noise levels (a) and (b), and two relative noise 
levels (c) and (d) are presented.}
   \label{fig:convergence}
\end{figure}

We tested the convergence behavior of {\sf STARS} with respect to dimension $n$ 
and noise levels on the same smooth convex
function $f_1$ with noise added in the same way as in
Section~\ref{sec.noise.invariant}. The results are summarized in
Figure~\ref{fig:convergence} , 
where (a) and (b) are for the additive case and (c) and (d) are for the
multiplicative case. 
The horizontal axis marks the problem dimension and the vertical axis shows the 
absolute accuracy. Two types of absolute accuracy are plotted. First, 
$\epsilon_{\mbox{pred}}$ (in blue $\times$'s) is the best achievable accuracy
given a 
certain noise level, computed by using (\ref{eq:predepsadd}) for the additive
case 
and (\ref{eq:predepsrel}) for the multiplicative case. Second is the actual 
accuracy (in red circle) achieved by {\sf STARS} after $N$ iterations where $N$, 
calculated as in (\ref{eq:N_add}), is the number of iterations needed in theory 
to get $\epsilon_{\mbox{pred}}$. 
Because of the stochastic nature of {\sf STARS}, we perform 15 runs (each with a
different random seed) of each test and report the averaged accuracy 
\begin{equation}\label{eqn.accuracy.formula}
\bar{\epsilon}_{\mbox{actual}} = \frac{1}{15}\sum_{i=1}^{15} 
\epsilon_{\mbox{actual}}^i = \frac{1}{15}\sum_{i=1}^{15} (f(x_{N}^i)-f^*).
\end{equation}



We observe from Figure \ref{fig:convergence} that the solution obtained by 
{\sf STARS} within the iteration limit $N$ is more accurate than that predicted
by 
the theoretical bounds. The difference between predicted and achieved accuracy
is always over an order of magnitude and is 
relatively consistent for all dimensions we examined. 

\subsection{Illustrative Example}\label{sec.illustrative.example}

In this section, we provide a comparison between {\sf STARS} 
and four other zero-order algorithms on noisy versions of 
(\ref{eqn.smooth.function}) with $n=8$. The methods we study all share a 
stochastic nature; that is, a random direction is generated at each iteration. 
Except for {\sf RP} \cite{Stich2011}, which is designed for
solving 
smooth convex functions, the rest are stochastic optimization algorithms. 
However, we still include {\sf RP} in the comparison because of its similar 
algorithmic framework. The algorithms and their function-specific 
inputs are summarized in Table \ref{tab:alg.param}, where 
$\tilde{L}_1$ and $\tilde{\sigma}^2$ are, respectively, estimations of $L_1$ and $\sigma^2$ given a noisy function (details on how to estimate $\tilde{L}_1$ and $\tilde{\sigma}^2$ are discussed in Appendix). 
We now briefly introduce each of the tested algorithms; 
algorithmic and implementation details are given in the appendix.

\begin{table}[htdp] 
\caption{Relevant function parameters for different methods.}
\begin{center}
\begin{tabular}{c|c|c}
\hline
 Method Abbreviation & Method Name & Parameters  \\  \hline
{\sf STARS} & Stepsize Approximation in Random Search & $L_1,\sigma^2$  \\ 
\hline
{\sf SS} & Random Search for Stochastic Optimization \cite{Nest2011} &$L_0, R^2$  \\  \hline
{\sf RSGF} & Random Stochastic Gradient Free method \cite{Lan2012} & $\tilde{L}_1,\tilde{\sigma}^2$ \\  \hline
{\sf RP} & Random Pursuit \cite{Stich2011} &  - \\  \hline
{\sf ES}& (1+1)-Evolution Strategy \cite{Schumer1968}& - \\  \hline
\end{tabular}
\end{center}
\label{tab:alg.param}
\end{table}%

The first zero-order method we include, named {\sf SS} (Random Search for 
Stochastic Optimization), is proposed in \cite{Nest2011} for solving 
(\ref{eqn.prob.chp3}). It assumes that $f \in \mathcal{C}^{0,0}(\R^n)$ is 
convex. The {\sf SS} algorithm, summarized in Algorithm \ref{alg:SS}, shares the 
same algorithmic framework as {\sf STARS} except for the choice of smoothing 
stepsize $\mu_k$ and the step length $h_k$. It is shown that the quantities  
$\mu_k$ and $h_k$ can be chosen so that a solution for (\ref{eqn.prob.chp3}) 
such that $f(x_N)-f^*\le \epsilon $ can be ensured by {\sf SS} in 
$\mathcal{O}(n^2/\epsilon^2)$ iterations. 

Another stochastic zero-order method that also shares an algorithmic 
framework similar to {\sf STARS} is {\sf RSGF} \cite{Lan2012}, which is
summarized in
Algorithm~\ref{alg:RSGF}. {\sf RSGF} targets the stochastic optimization
objective function in 
(\ref{eqn.prob.chp3}), but the authors relax the convexity assumption and 
allow $f$ to be nonconvex. However, it is assumed that $ 
\tilde{f}(\cdot,\xi)\in  \mathcal{C}^{1,1}(\mathbb{R}^n)  $ almost surely, which 
implies that $f\in  \mathcal{C}^{1,1}(\mathbb{R}^n)$. The authors show that the 
iteration complexity for {\sf RSGF} finding an $\epsilon$-accurate solution,
(i.e., a point 
$\bar{x}$ such that $\mathbb{E}[\| \nabla f(\bar{x}) \|] \le \epsilon$) can be 
bounded by $\mathcal{O}(n/\epsilon^2)$. Since such a solution $\bar{x}$ 
satisfies $f(\bar{x})-f^*\le \epsilon $ when $f$ is convex, this bound improves 
Nesterov's result in \cite{Nest2011} by a factor $n$ for convex 
stochastic optimization problems. 

In contrast with the presented randomized approaches that work with a Gaussian
vector $u$, we include an algorithm that samples from a 
uniform distribution on the unit hypersphere. Summarized in
Algorithm~\ref{alg:RP}, {\sf RP} \cite{Stich2011} is designed for unconstrained,
smooth, convex
optimization.  
It relaxes the requirement in \cite{Nest2011} of approximating directional 
derivatives via a suitable oracle. Instead, the sampling directions are chosen
uniformly at random on the unit hypersphere, and the step lengths are determined
by a 
line search oracle. This randomized method also requires only zeroth-order 
information about the objective function, but it does not need any 
function-specific parametrization. It was shown that {\sf RP} meets the 
convergence rates of the standard steepest descent method up to a factor $n$. 

Experimental studies of variants of $ (1+1)$-Evolution Strategy ({\sf ES}), 
first proposed by Schumer and Steiglitz \cite{Schumer1968}, have shown their 
effectiveness in practice and their robustness in noisy environment. However,
provable 
convergence rates are derived only for the simplest forms of {\sf ES} on 
unimodal objective functions \cite{jagerskupper20061+,auge2005a,jah:2010a}, such 
as sphere or ellipsoidal functions. 
The implementation we study is summarized in Algorithm \ref{alg:ES}; however, 
different variants of this scheme have been studied in 
\cite{beyer2002evolution}.

\begin{figure}[t!]
     \begin{center}
        \subfigure[$\sigma_a=10^{-5}$]{%
            \label{fig:traj.abs.5}
            \includegraphics[width=0.5\textwidth]{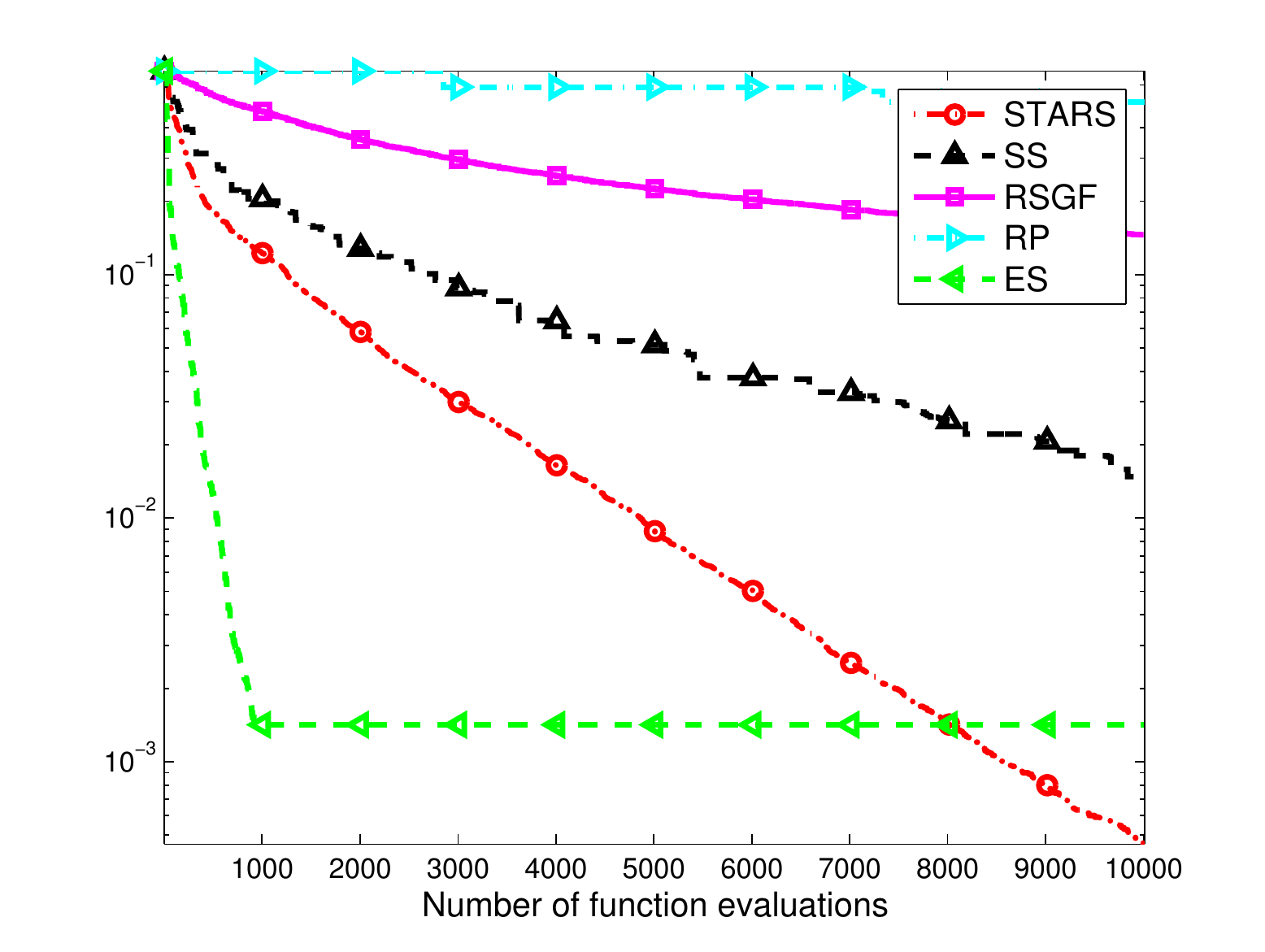}
        }%
        \subfigure[$\sigma_r=10^{-5}$]{%
           \label{fig:traj.rel.5}
           \includegraphics[width=0.5\textwidth]{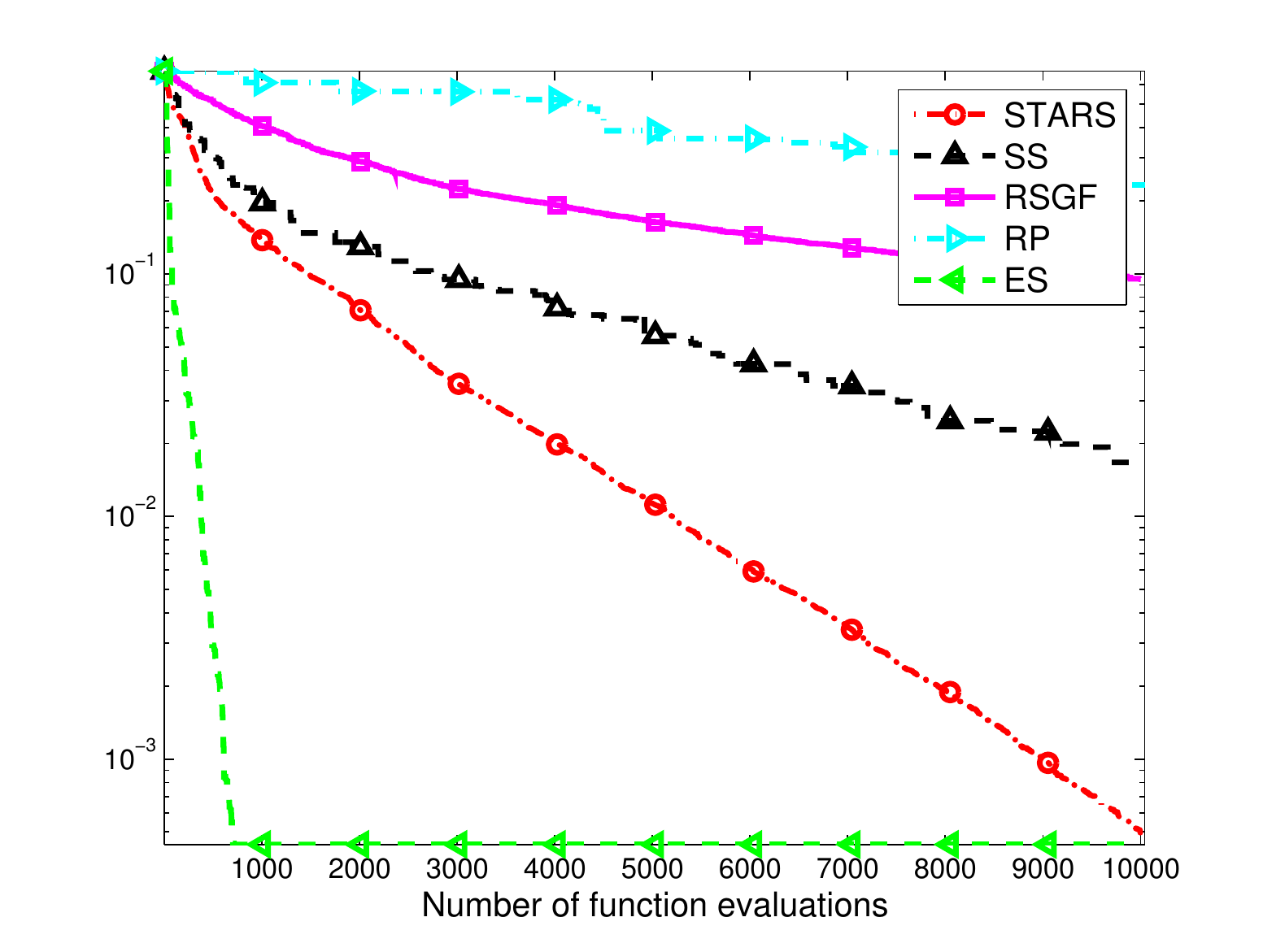}
        }\\
                \subfigure[$\sigma_a=10^{-3}$]{%
            \label{fig:traj.abs.3}
            \includegraphics[width=0.5\textwidth]{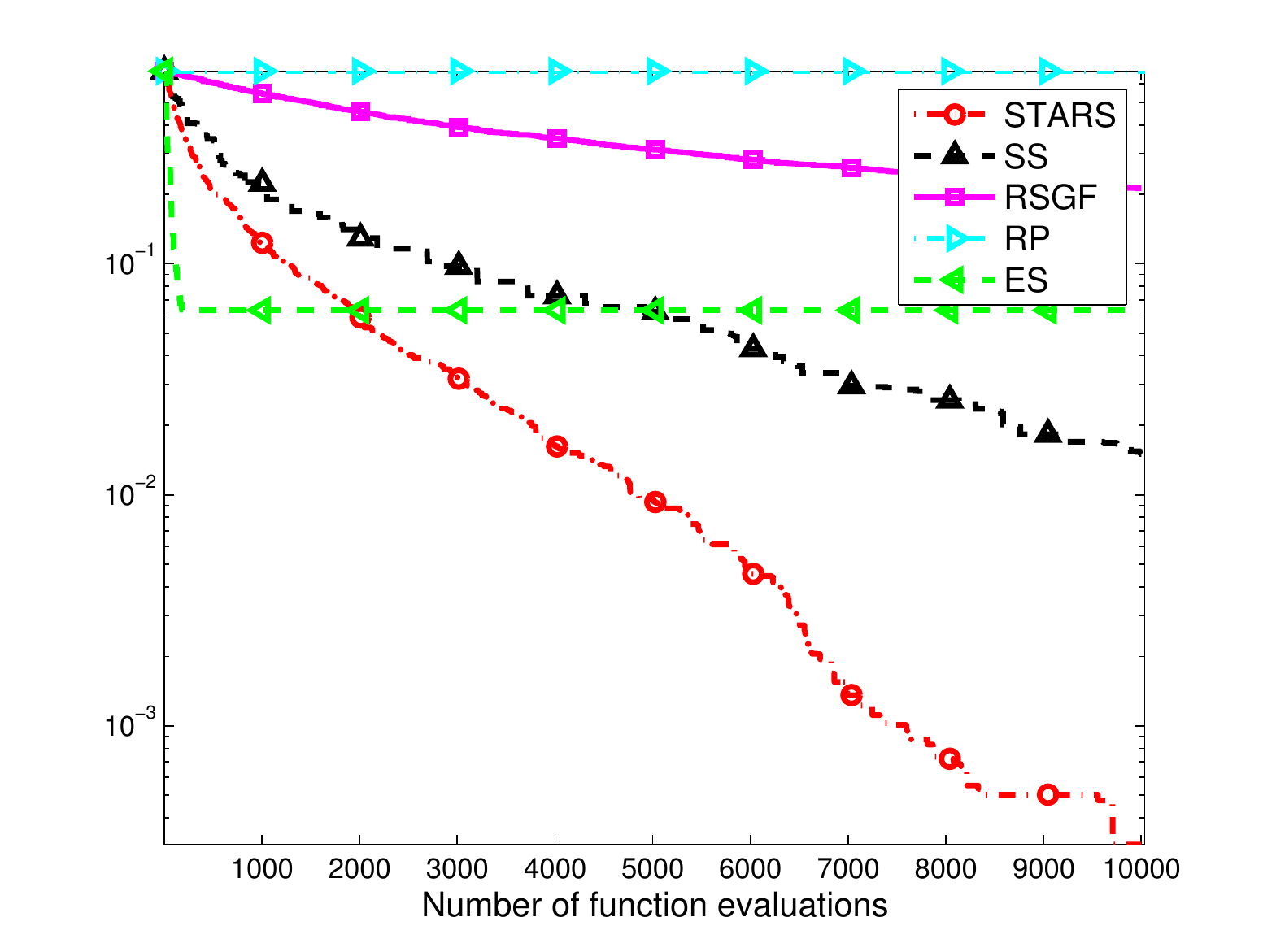}
        }%
        \subfigure[$\sigma_r=10^{-3}$]{%
           \label{fig:traj.rel.3}
           \includegraphics[width=0.5\textwidth]{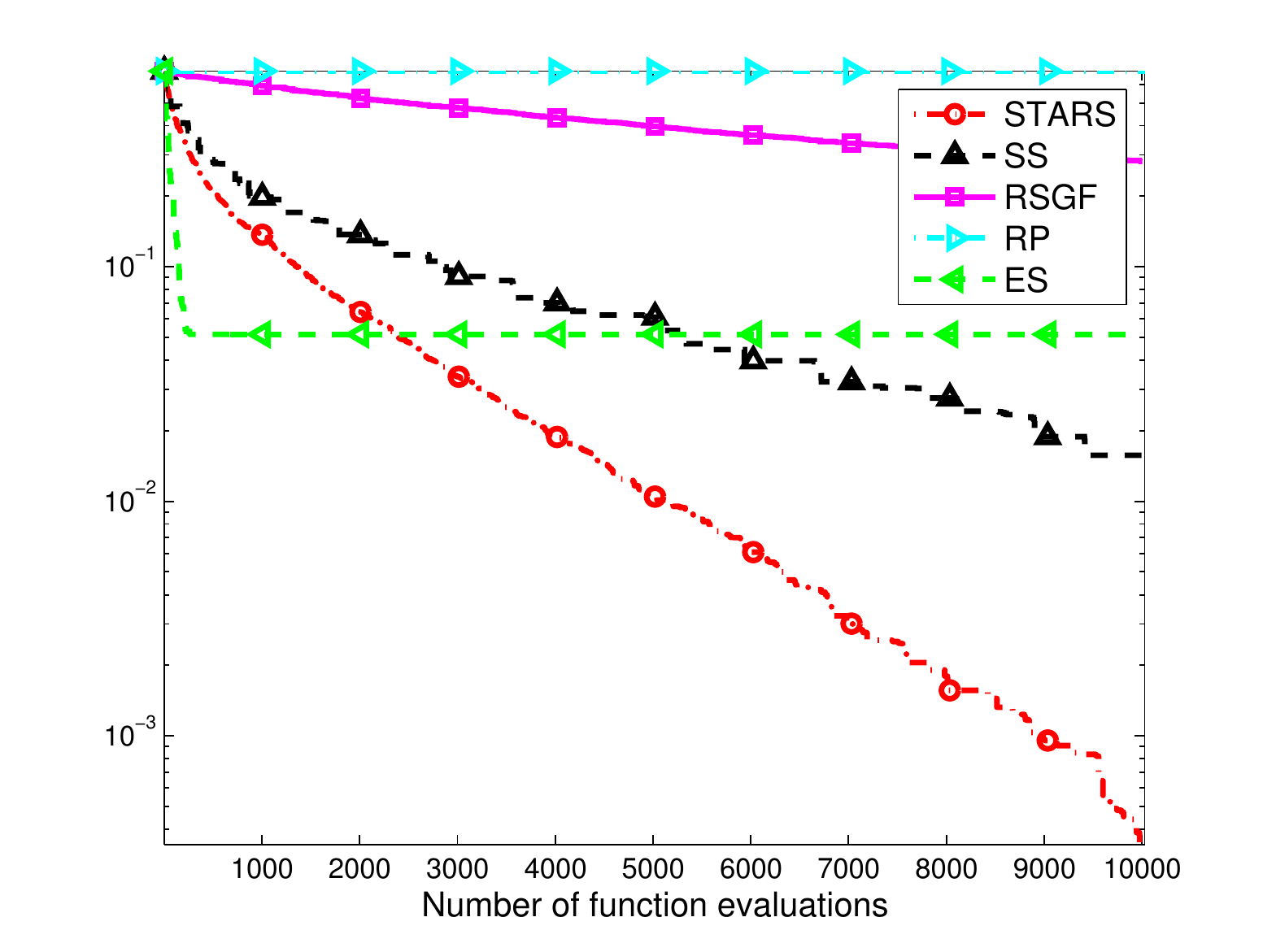}
        }\\ß
                \subfigure[$\sigma_a=10^{-1}$]{%
            \label{fig:traj.abs.1}
            \includegraphics[width=0.5\textwidth]{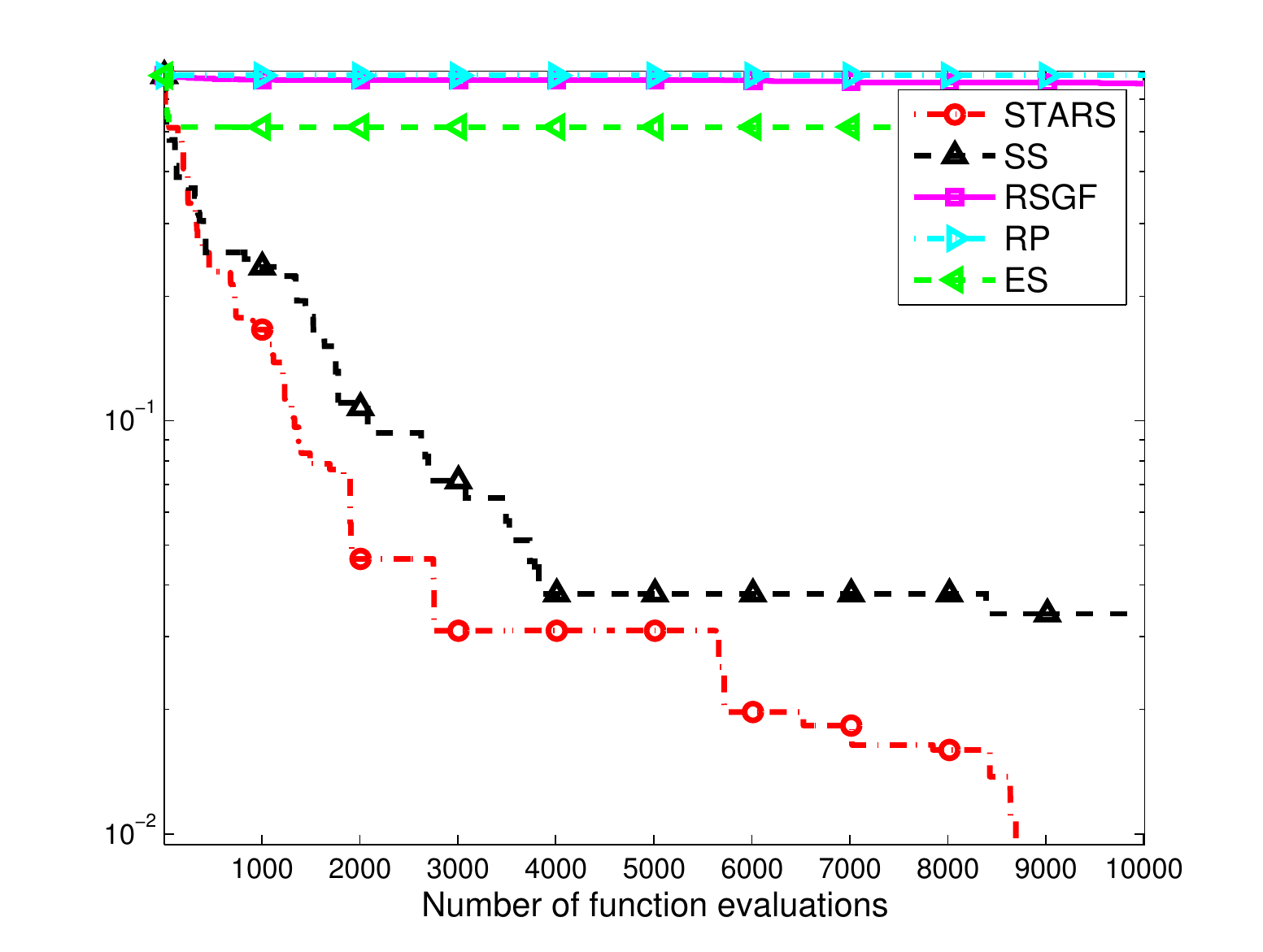}
        }%
        \subfigure[$\sigma_r=10^{-1}$]{%
           \label{fig:traj.rel.1}
           \includegraphics[width=0.5\textwidth]{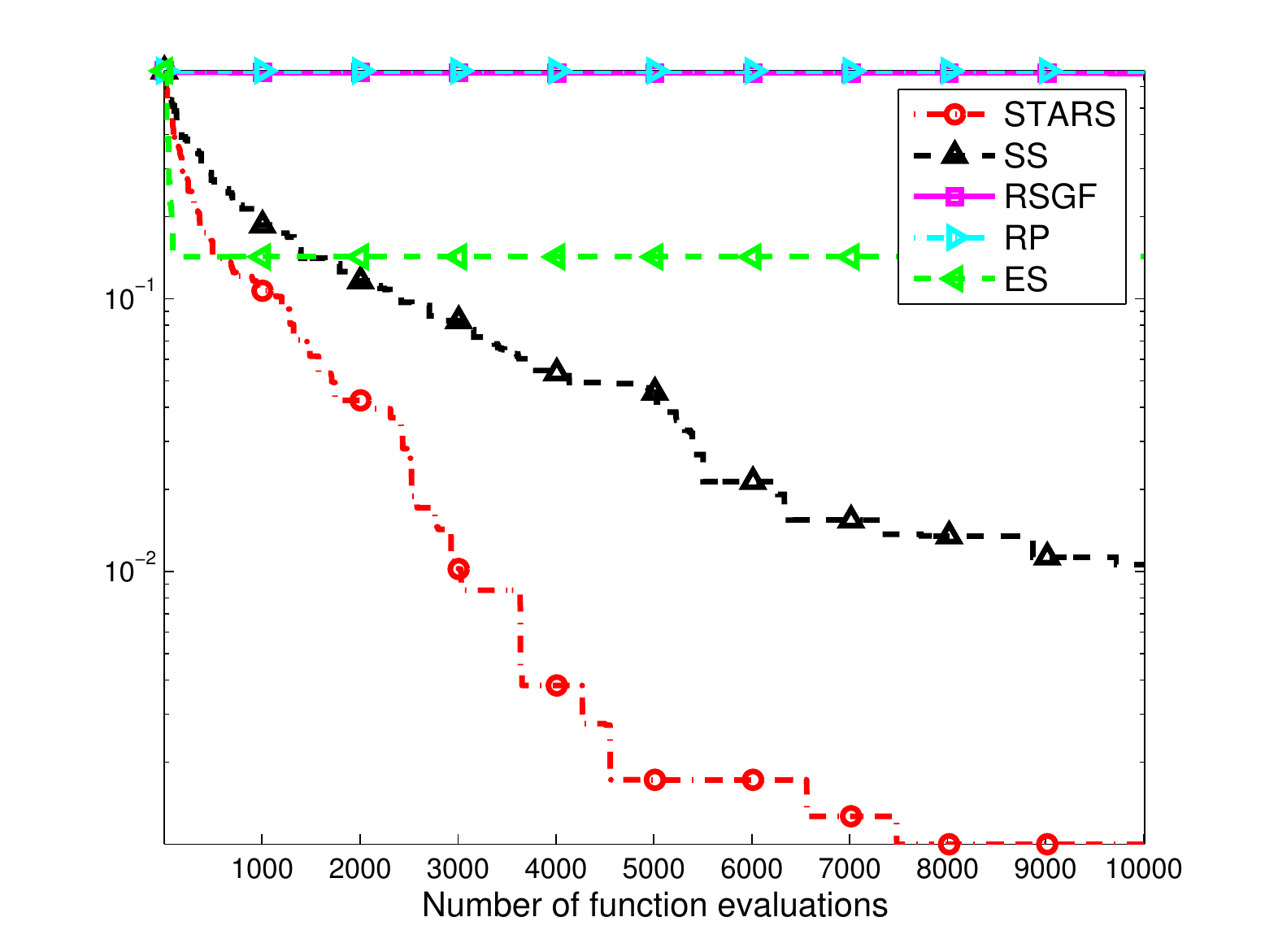}
        }
            \end{center}
            \vspace{-0.2in}
    \caption{Trajectory plots of five zero-order methods in the 
additive and multiplicative noise settings. The vertical axis represents the
true function value 
$f(x_k)$, and each line is the mean of 20 trials.\label{fig:traj.all}}
\end{figure}

We observe from Figure \ref{fig:traj.all} that {\sf STARS} outperforms the 
other four algorithms in terms of final accuracy in the solution. In both 
Figures~\ref{fig:traj.all}(a)~and~\ref{fig:traj.all}(b), {\sf ES} is the
fastest algorithm 
among all in the beginning. However, {\sf ES}  stops progressing after a few
iterations, 
whereas {\sf STARS} keeps progressing to a more accurate solution. As the noise 
level increases from $10^{-5}$ to $10^{-1}$, the performance of {\sf ES} 
gradually worsens, similar to the other methods {\sf SS}, {\sf RSGF}, and
{\sf 
RP}. However, the noise-invariant property of {\sf STARS} allows it to remain 
robust in these noisy environments.

\section*{Acknowledgments}
We are grateful to Katya Scheinberg for valuable discussions.

\clearpage

\section{Appendix}
\label{sec:append}
In this appendix we describe the implementation details of the four 
zero-order methods tested in Table~\ref{tab:alg.param} and
Section~\ref{sec.illustrative.example}. 

\subsection*{Random Search for Stochastic Optimization}
\balgorithm
\caption{({\sf SS}: Random Search for Stochastic Optimization)}
\label{alg:SS}
 \balgorithmic[1]
\STATE Choose initial point $x_0$ and iteration limit $N$. Fix step length $h_k 
= h = \frac{R}{(n+4)(N+1)^{1/2}L_0 }$ and smoothing stepsize $ \mu_k= \mu  
= \frac{\epsilon}{2 L_0 n^{1/2}}$. Set $k \gets 1$.\\
\STATE Generate a random Gaussian vector $u_k$. \\
\STATE Evaluate the function values $\tilde{f}(x_k;\xi_k)$ and $\tilde{f}(x_k + 
\mu_ku_k;\xi_k)$.\\
\STATE Call the random stochastic gradient-free oracle \\
$$s_\mu (x_k;u_k,\xi_k) = 
\frac{\tilde{f}(x_k+\mu_ku_k;\xi_k)-\tilde{f}(x_k;\xi_{k})}{\mu_k}u_k.$$
\STATE Set  $x_{k+1}=x_k-h_k s_\mu (x_k;u_k,\xi_k) $, update $k\gets k+1$, and 
return to Step 2.
  \ealgorithmic
\ealgorithm

Algorithm~\ref{alg:SS} provides the {\sf SS} (Random Search for Stochastic 
Optimization) algorithm from \cite{Nest2011}.

\paragraph{Remark:} $\epsilon$ is suggested to be $2^{-16}$ in the experiments 
in \cite{Nest2011}. Our experiments in Section~\ref{sec.illustrative.example},
however, show that this choice of $\epsilon$ forces $\sf SS$ to take small
steps
and thus  $\sf SS$ does not converge at all in the noisy environment. Hence, we
increase $\epsilon$ (to $\epsilon=0.1$) to show that optimistically, {\sf SS}
will work if the stepsize is big enough. Although in the
additive noise case one can recover {\sf STARS} by appropriately setting this
$\epsilon$ in {\sf SS}, it is not possible in the multiplicative case 
because {\sf STARS} takes dynamically adjusted smoothing stepsizes in this case.


\subsection*{Randomized Stochastic Gradient-Free Method}

\balgorithm
\caption{({\sf RSGF}: Randomized Stochastic Gradient-Free Method)}
\label{alg:RSGF}
 \balgorithmic[1]
\STATE Choose initial point $x_0$ and iteration limit $N$. Estimate $L_1$ and 
$\tilde{\sigma}^2$ of the noisy function $\tilde{f}$. Fix step length as 
$$\gamma_k= \gamma = \dfrac{1}{\sqrt{n+4}}  \min \left\{  
\dfrac{1}{4L_1\sqrt{n+4}}, \dfrac{\tilde{D}}{\tilde{\sigma}\sqrt{N}}  
\right\},$$ 
where 
$\tilde{D}=(2f(x_0)/L_1)^{\frac{1}{2}}$. Fix $ \mu_k = \mu = 0.0025$. Set $k 
\gets 1$.\\
\STATE Generate a Gaussian vector $u_k$. \\
\STATE Evaluate the function values $\tilde{f}(x_k;\xi_k)$ and $\tilde{f}(x_k + 
\mu_ku_k;\xi_k)$.\\
\STATE Call the stochastic zero-order oracle \\
$$G_\mu (x_k;u_k,\xi_k) = 
\frac{\tilde{f}(x_k+\mu_ku_k;\xi_k)-\tilde{f}(x_k;\xi_{k})}{\mu}u_k.$$
\STATE Set  $x_{k+1}=x_k-\gamma_k G_\mu (x_k;u_k,\xi_k) $, update $k\gets k+1$, 
and return to Step 2.
  \ealgorithmic
\ealgorithm

Algorithm~\ref{alg:RSGF} provides the {\sf RSGF} (Randomized Stochastic 
Gradient-Free Method) algorithm from \cite{Lan2012}.

\paragraph{Remark:} Although the convergence analysis of {\sf RSGF} is based on
knowledge of the constants $\L_1$ and $\sigma^2$, the discussion in
\cite{Lan2012} on how to implement {\sf RSGF} does not reply on these
inputs. Because the authors solved a support vector machine problem and an
inventory problem, both of which do not have known $L_1$ and $\sigma^2$ values,
they provide details on how to estimate these parameters given a noisy
function. Hence following \cite{Lan2012}, the parameter $L_1$ is estimated as
the $l_2$ norm of the Hessian of the 
deterministic approximation of the noisy objective functions. This estimation is
achieved 
by using a sample average approximation approach with 200 i.i.d.~samples. Also,
we compute the 
stochastic gradients of the objective functions at these randomly selected 
points and take the maximum variance of the stochastic gradients as an 
estimate of $\tilde{\sigma}^2$. 


\subsection*{Random Pursuit}

\balgorithm
\caption{({\sf RP}: Random Pursuit)}
\label{alg:RP}
 \balgorithmic[1]
\STATE Choose initial point $x_0$, iteration limit $N$, and line search 
accuracy $ \mu = 0.0025$. Set $k \gets 1$.\\
\STATE Choose a random Gaussian vector $u_k$. \\
\STATE Choose $x_{k+1}=x_k+{\sf LS_{APPROX_\mu}}(x_k,u_k)\cdot u_k $, update 
$k\gets k+1$, and return to Step 2.
  \ealgorithmic
\ealgorithm

Algorithm~\ref{alg:RP} provides the {\sf RP} (Random Pursuit) 
algorithm from \cite{Stich2011}.

\paragraph{Remark:} We follow the authors in 
\cite{Stich2011} and use the built-in MATLAB routine 
$\mathsf{fminunc.m}$ as the approximate line search oracle.

\subsection*{$ (1+1)$-Evolution Strategy}
\balgorithm
\caption{({\sf ES}: $ (1+1)$-Evolution Strategy)}
\label{alg:ES}
 \balgorithmic[1]
\STATE Choose initial point $x_0$, initial stepsize $\sigma_0$, iteration limit 
$N$, and probability of improvement $p=0.27$. Set $c_s=e^{\frac{1}{3}}\approx 
1.3956$ and $c_f = c_s \cdot e^{\frac{-p}{1-p}}\approx 0.8840$. Set $k \gets 
1$.\\
\STATE Generate a random Gaussian vector $u_k$.\\
\STATE Evaluate the function values $\tilde{f}(x_k;\xi_k)$ and $\tilde{f}(x_k + 
\sigma_k u_k;\xi_k)$.\\
\STATE If $\tilde{f}(x_k + \sigma_k u_k;\xi_k)\le\tilde{f}(x_k;\xi_k)$, then 
set $x_{k+1}=x_k +\sigma_k u_k$ and $\sigma_{k+1}=c_s\sigma_k$; 
\\ 
Otherwise, set $x_{k+1}=x_k$ and $\sigma_{k+1}=c_f\sigma_k$.
\STATE Update $k\gets k+1$ and return to Step 2.
  \ealgorithmic
\ealgorithm

Algorithm~\ref{alg:ES} provides the {\sf ES} ($ 
(1+1)$-Evolution Strategy) algorithm from \cite{Schumer1968}.

\paragraph{Remark:} A problem-specific parameter required by Algorithm 
\ref{alg:ES} is the initial stepsize $\sigma_0$, which is given in 
\cite{Stich2011} for some of our test functions. The stepsize is multiplied by 
a factor $c_s=e^{1/3}>1$ when the mutant's fitness is as good as the parent is  
and is otherwise multiplied by $c_s \cdot e^{\frac{-p}{1-p}}<1$, where $p$ is 
the probability of improvement set to the value $0.27$ suggested by Schumer 
and Steiglitz \cite{Schumer1968}.


\bibliographystyle{siam} 
\bibliography{noisecited}

\begin{thebibliography}{10}

\bibitem{MAAbramson_CAudet_2006}
{\sc M.~A. Abramson and C.~Audet}, {\em Convergence of mesh adaptive direct
  search to second-order stationary points}, SIAM Journal on Optimization, 17
  (2006), pp.~606--619.

\bibitem{MAAbramson_etal_2008}
{\sc M.~A. Abramson, C.~Audet, J.~E. {Dennis, Jr.}, and S.~{Le Digabel}}, {\em
  {OrthoMADS: A} deterministic {MADS} instance with orthogonal directions},
  SIAM Journal on Optimization, 20 (2009), pp.~948--966.

\bibitem{NIPS2011_4475}
{\sc Alekh Agarwal, Dean~P. Foster, Daniel~J. Hsu, Sham~M. Kakade, and
  Alexander Rakhlin}, {\em Stochastic convex optimization with bandit
  feedback}, in Advances in Neural Information Processing Systems 24, 2011,
  pp.~1035--1043.

\bibitem{CAudet_JEDennis_2006}
{\sc C.~Audet and J.~E. {Dennis, Jr.}}, {\em Mesh adaptive direct search
  algorithms for constrained optimization}, SIAM Journal on Optimization, 17
  (2006), pp.~188--217.

\bibitem{auge2005a}
{\sc A.~Auger}, {\em Convergence results for the $(1,\lambda)$-{SA-ES} using
  the theory of $\varphi$-irreducible {Markov} chains}, Theoretical Computer
  Science, 334 (2005), pp.~35--69.

\bibitem{beyer2002evolution}
{\sc Hans-Georg Beyer and Hans-Paul Schwefel}, {\em Evolution strategies-- {A}
  comprehensive introduction}, Natural Computing, 1 (2002), pp.~3--52.

\bibitem{Lan2012}
{\sc S.~Ghadimi and G.~Lan}, {\em Stochastic first- and zeroth-order methods
  for nonconvex stochastic programming}, SIAM Journal on Optimization, 23
  (2013), pp.~2341--2368.

\bibitem{jagerskupper20061+}
{\sc Jens J{\"a}gersk{\"u}pper}, {\em How the (1+1)-{ES} using isotropic
  mutations minimizes positive definite quadratic forms}, Theoretical Computer
  Science, 361 (2006), pp.~38--56.

\bibitem{jah:2010a}
{\sc M.~Jebalia, A.~Auger, and N.~Hansen}, {\em Log-linear convergence and
  divergence of the scale-invariant (1+1)-{ES} in noisy environments},
  Algorithmica, 59 (2011), pp.~425--460.

\bibitem{RMLewis_VTorczon_MTrosset_2000}
{\sc R.~M. Lewis, V.~Torczon, and M.~Trosset}, {\em Direct search methods:
  {Then} and now}, Journal of Computational and Applied Mathematics, 124
  (2000), pp.~191--207.

\bibitem{matyas1965}
{\sc J.~Matyas}, {\em Random optimization}, Automation and Remote Control, 26
  (1965), pp.~246--253.

\bibitem{JMSMW11}
{\sc Jorge~J. Mor\'e and Stefan~M. Wild}, {\em Estimating computational noise},
  SIAM Journal on Scientific Computing, 33 (2011), pp.~1292--1314.

\bibitem{More2012}
{\sc Jorge~J. Mor{\'e} and Stefan~M. Wild}, {\em Estimating derivatives of
  noisy simulations}, ACM Transactions on Mathematical Software, 38 (2012),
  pp.~19:1--19:21.

\bibitem{more2014nd}
{\sc Jorge~J. Mor\'e and Stefan~M. Wild}, {\em Do you trust derivatives or
  differences?}, Journal of Computational Physics, 273 (2014), pp.~268--277.

\bibitem{Nest2011}
{\sc Yurii Nesterov}, {\em Random gradient-free minimization of convex
  functions}, CORE Discussion Papers 2011001, Universit{\'e} Catholique de
  Louvain, Center for Operations Research and Econometrics (CORE), 2011.

\bibitem{polyak1987}
{\sc B.T. Polyak}, {\em Introduction to Optimization}, Optimization Software,
  1987.

\bibitem{NIPS2012_4509}
{\sc Ben Recht, Kevin~G. Jamieson, and Robert Nowak}, {\em Query complexity of
  derivative-free optimization}, in Advances in Neural Information Processing
  Systems 25, 2012, pp.~2672--2680.

\bibitem{Schumer1968}
{\sc M.~Schumer and K.~Steiglitz}, {\em Adaptive step size random search}, IEEE
  Transactions on Automatic Control, 13 (1968), pp.~270--276.

\bibitem{Stich2011}
{\sc Sebastian~U. Stich, Christian~L. M{\"{u}}ller, and Bernd G{\"{a}}rtner},
  {\em Optimization of convex functions with random pursuit}, SIAM Journal on
  Optimization, 23 (2013), pp.~1284--1309.

\bibitem{VTorczon_1991}
{\sc V.~Torczon}, {\em On the convergence of the multidirectional search
  algorithm}, SIAM Journal on Optimization, 1 (1991), pp.~123--145.

\bibitem{VTorczon_1997}
{\sc {V. Torczon}}, {\em On the convergence of pattern search algorithms}, SIAM
  Journal on Optimization, 7 (1997), pp.~1--25.

\end{thebibliography}

\vspace{3em}

\small


\end{document}